\documentclass[12pt]{article}
\usepackage[utf8]{inputenc}

\usepackage[left=1cm, top=1cm, right=1cm, bottom=1cm, includeheadfoot, headheight=0pt, a4paper]{geometry}
\usepackage{fancyhdr}
\usepackage{lastpage}
\usepackage{amsmath}
\usepackage{amssymb}
\usepackage{graphicx}
\usepackage{enumitem}
\usepackage{multicol}
\usepackage{amsmath}
\usepackage{amsfonts}
\usepackage[utf8]{inputenc}
\usepackage{amsthm}
\usepackage{tikz-cd}
\usepackage{geometry}
\usepackage{mathtools}
\usepackage{tikz}
\usepackage{comment}
\usepackage{url}
\usetikzlibrary{patterns}
\usepackage{svg}
\usepackage{xcolor}
\usepackage[T1]{fontenc}
\usepackage[utf8]{inputenc}
\usepackage{float}
\usepackage{subcaption}

\newtheorem{thm}{Theorem}[section]
\newtheorem{lemma}[thm]{Lemma}

\newtheorem{prop}[thm]{Proposition}
\newtheorem{cor}[thm]{Corollary}

\newtheorem{conj}[thm]{Conjecture}
\newtheorem{prob}[thm]{Problem}
\theoremstyle{remark}
\newtheorem{remark}[thm]{Remark}
\newtheorem{ex}[thm]{Example}
\renewcommand{\thefootnote}{\fnsymbol{footnote}}

\newcommand\blfootnote[1]{%
  \begingroup
  \renewcommand\thefootnote{}\footnote{#1}%
  \addtocounter{footnote}{-1}%
  \endgroup
}

\title{The type and cardinality of minimal presentations of \\ numerical semigroups with embedding dimension four }

\author{Kazimierz Chomicz}
\date{September 2026}
\begin{document}

\maketitle

\abstract{ 
For a numerical semigroup $S$ with embedding dimension four, we study the relationship between its type $t(S)$ and the cardinality of its minimal presentations $\eta(S)$. Using an approach based on the geometry of the Ap{\'e}ry set, we prove that $4t(S) + 5 \geq \eta(S) \geq t(S)-11$. This resolves a problem of Moscariello and Sammartano asking whether $t(S)$ is bounded by a function of $\eta(S)$, and improves the previously known bound $9t(S) + 4 \geq \eta(S)$ due to Bresinsky.
}

\section{Introduction}

\blfootnote{\text{Keywords: numerical semigroup, embedding dimension, Ap{\'e}ry set, type, minimal presentation.}}
\blfootnote{\text{2020 Mathematics Subject Classification: 20M13, 20M14}}

Let $\mathbb{N}$ denote the set of nonnegative integers. For $k> 0$ and $d_0 ,d_1, \dots, d_k \in \mathbb{N}$ such that $\gcd(d_0,d_1, \dots, d_k) = 1$, the set
$$ S = \langle d_0,d_1, \dots, d_k \rangle  = \{ \lambda_0 d_0 + \lambda_1 d_1 + \dots + \lambda_k d_k \mid \lambda_i \in \mathbb{N}\},$$
is called a \emph{numerical semigroup} generated by $\{d_0,d_1,\dots,d_k\}$. Equivalently, $S$ is a submonoid of $(\mathbb{N},+)$ such that $\mathbb{N} \setminus S$ is finite (see \cite[Lemma 2.1]{Rosales2009}). If no proper subset of $\{d_0,d_1,\dots,d_k\}$ generates $S$, then we say that $S$ is minimally generated by $\{d_0,d_1,\dots,d_k\}$. The cardinality of the minimal set of generators of $S$ is called the \emph{embedding dimension} of $S$ and is denoted $e(S)$. 

The largest integer that is not in $S$ is called the \emph{Frobenius number} and is denoted $F(S)$. An integer $a \notin S$ is called a \emph{pseudo-Frobenius number} of $S$ if $a + s \in S$ for all $s \in S \setminus \{0\}$. The set of all pseudo-Frobenius numbers of $S$ is denoted $PF(S)$. The cardinality of $PF(S)$ is called the \emph{type} of $S$ and is denoted $t(S)$. 

If $S$ is minimally generated by $\{d_0,d_1,\dots,d_k\}$, then we have a homomorphism $\varphi : \mathbb{N}^{k+1} \to S$ defined by $\varphi(z_0, \dots, z_k) = z_0d_0 + \dots + z_kd_k$. 
The \emph{kernel congruence} of $\varphi$ is $\ker(\varphi) = \{ (a,b) \in \mathbb{N}^{k+1} \times \mathbb{N}^{k+1} \mid \varphi(a) = \varphi(b) \}$. 
It is well known that $S$ is isomorphic to the monoid $\mathbb{N}^{k+1} / \ker(\varphi)$. A minimal set of generators of the congruence $\ker(\varphi)$ is called a \emph{minimal presentation} of $S$. Minimal presentations of $S$ are not unique, but their cardinality depends only on $S$ and is denoted by $\eta(S)$. 

Both type and minimal presentations are central to the study of numerical semigroups (see \cite{AssiDAnnaGarciaSanchez2020,BarucciDobbsFontana1997,Froberg1987,MoscarielloSammartano2025,Rosales2009}).
In the study of minimal presentations, the invariant $\eta(S)$ has attracted considerable interest (see \cite{Elmacioglu2024,MoscarielloSammartanominpresen2025,MoscarielloSammartano2025,Rosales2009,Stamate2018}). 

A numerical semigroup $S$ is \emph{symmetric} if for every $x \in \mathbb{Z} \setminus S$ we have $F(S) - x \in S$. This is equivalent to $t(S)=1$ \cite{Kunz1970}. It is well known that if $e(S)=2$, then $t(S) = \eta(S)=1$. The case $e(S)=3$ is completely settled in the following theorem.

\begin{thm}\label{theorem_embd3}
    \textup{\cite{Herzog1970}} Let $S$ be a numerical semigroup with $e(S)=3$. If $S$ is symmetric, then $\eta(S)=2$ and $t(S)=1$. If $S$ is not symmetric, then $\eta(S)=3$ and $t(S)=2$.
\end{thm}

When $e(S) \geq 4$, the situation becomes much more complicated. Bresinsky \cite{Bresinsky1975onprimeideals} found a family of numerical semigroups with $e(S) = 4$ whose $\eta(S)$ can be arbitrarily large, while Backelin \cite[p.\ 75]{Froberg1987} found a family of numerical semigroups with $e(S) = 4$ whose $t(S)$ can be arbitrarily large.

In this paper, we study $t(S)$ and $\eta(S)$ for numerical semigroups with embedding dimension four. In \cite{Bresinsky1988}, Bresinsky established the following bound.

\begin{thm}\label{theoremtypebre}
    \textup{\cite[Theorem 7]{Bresinsky1988}} If $S$ is a numerical semigroup with $e(S) = 4$, then $  9 t(S) + 4 \geq \eta(S)$.
\end{thm}

In \cite{MoscarielloSammartano2025}, the authors, inspired by the above bound, proposed the following problem.

\begin{prob}\label{problemtype}
    \textup{\cite[Problem 25]{MoscarielloSammartano2025}} Let $S$ be a numerical semigroup with $e(S) = 4$. Is $t(S)$ bounded by a function of $\eta(S)$?
\end{prob}

Resolving Problem \ref{problemtype} also settles a general
problem on the boundedness of Betti numbers \cite[Problem 24]{MoscarielloSammartano2025} in the case $e(S)=4$, as noted in the same paper.

The main result of this paper is the following theorem, which settles Problem \ref{problemtype} and improves the bound of Theorem \ref{theoremtypebre}.

\begin{thm}\label{theorem_type_main}
    If $S$ is a numerical semigroup with $e(S) = 4$, then $4 t(S) +5 \geq \eta(S) \geq t(S) - 11$. 
\end{thm}

The paper is organized as follows.
In Section \ref{section_geometricprocedure}, we study the geometry of the Ap{\'e}ry set of numerical semigroups with embedding dimension four using $L$-shapes \cite{AGUILOGOST2015}. 
In Section \ref{section_minimal_presentations}, we relate minimal presentations of numerical semigroups with embedding dimension four to an $L$-shape.
In Section \ref{section_type}, we prove Theorem \ref{theorem_type_main} by using results of the previous sections to study $t(S)$ and $\eta(S)$ geometrically.
We conclude the paper with Conjecture \ref{conjecture}.

\section{Geometry of the Ap{\'e}ry set}\label{section_geometricprocedure}

In this section, we study the geometry of the Ap{\'e}ry set of numerical semigroups with $e(S)=4$. Subsection~\ref{subsection_minimal_relations} introduces the notion of minimal relations and uses them to classify numerical semigroups with $e(S)=4$ into primary and secondary semigroups (Lemma~\ref{lemma2cases}) --- a distinction that underpins our approach. Subsection~\ref{subsection_framework} presents the geometric framework for analyzing the Apéry set (Theorem~\ref{theoremprocedure}). Subsection~\ref{subsection_Lshapes} introduces the notion of $L$-shapes and examines them for primary and secondary semigroups (Propositions \ref{prop_primary} and \ref{prop_secondary}).

Let $S$ be minimally generated by $\{d_0,d_1,\dots,d_k\}$. For an element $s \in S$, the set
\begin{equation}
\varphi^{-1}(s) = \{(z_0,\dots,z_k)\in \mathbb{N}^{k+1} \mid z_0d_0 + \dots + z_kd_k = s\}
\end{equation}
is called the \emph{set of factorizations} of $s$.
For $z =(z_0,\dots,z_k) \in \varphi^{-1}(s)$, the \emph{support} of $z$ is 
\begin{equation}
\text{supp}(z) = \{i \in \{0,\dots,k\} \mid z_i > 0\}.
\end{equation}
For an element $a \in S$, the set
\begin{equation}
\textup{Ap}(S,a) = \{ s \in S \mid s-a \notin S\}
\end{equation}
is called the \emph{Ap{\'e}ry set} of $S$ with respect to $a$. Note that $\textup{Ap}(S,a)$ is the set of the smallest nonnegative integers from each residue class modulo $a$ that are in $S$, hence $ |\textup{Ap}(S,a)| = a$. From \cite[Proposition~2.20]{Rosales2009}, we have
\begin{equation}\label{eq_pseudo_frob}
     PF(S) = \{s - a \mid s \in \textup{Ap}(S,a), s + d_i \notin \textup{Ap}(S,a) \text{ for } i = 0,1,\dots,k\}.
\end{equation}
This description of pseudo-Frobenius numbers in terms of the Ap{\'e}ry set will provide a useful framework for our study of the type in Section \ref{section_type}.

\subsection{Minimal relations}\label{subsection_minimal_relations}

Let $S =\langle d_0,d_1,d_2,d_3 \rangle$ be a numerical semigroup with embedding dimension four.

For every $i \in \{0,1,2,3\}$, let $a_{ii}$ be the smallest positive integer such that there exist nonnegative integers $a_{ij},a_{ik},a_{il}$ satisfying
\begin{equation}\label{minimalrelation_eq1}
     a_{ii} d_i = a_{ij} d_j + a_{ik}d_k + a_{il} d_l
\end{equation}
where $\{i,j,k,l\} = \{0,1,2,3\}$. We call \eqref{minimalrelation_eq1} a \emph{minimal relation} of $d_i$.

Fix $m \in \{0,1,2,3\}$. For every $i \in \{0,1,2,3\} \setminus \{m\}$, let $b_{m,ii}$ be the smallest positive integer such that there exist nonnegative integers $b_{m,ij},b_{m,ik},b_{m,il}$ satisfying 
\begin{equation}\label{minimalrelation_eq2}
     b_{m,ii} d_i = b_{m,ij} d_j + b_{m,ik}d_k + b_{m,il} d_l
\end{equation}
where $\{i,j,k,l\} = \{0,1,2,3\}$ and $b_{m,im} > 0$. We call \eqref{minimalrelation_eq2} a \emph{$d_m$-positive minimal relation of $d_i$}.

Note that in the minimal relations $a_{ii}$ is unique; however, each $a_{ij}$ for $i \neq j$ is not necessarily unique. Similarly, in the $d_m$-positive minimal relations $b_{m,ii}$ is unique; however, each $b_{m,ij}$ for $i \neq j$ is not necessarily unique.

Let $\varphi : \mathbb{N}^{4} \to S$ be the factorization morphism of $S$. 

\begin{lemma}\label{lemma_catenary_minrelations}
    For $i \in \{0,1,2,3\}$ and $(\mu_0,\mu_1,\mu_2,\mu_3) \in \varphi^{-1}( a_{ii}d_i)$, we have either $\mu_i = a_{ii}$ and $\mu_j = 0$ for all $j \in \{0,1,2,3\} \setminus \{i\}$, or $\mu_i = 0$.
\end{lemma}

\begin{proof}
     Without loss of generality, suppose $i=0$. We have
     \begin{equation}\label{lemma_catenary2_eq1}
         (a_{00}-\mu_0)d_0 =  \mu_1d_1 + \mu_2d_2 + \mu_3d_3.
     \end{equation}
     The right-hand side is clearly nonnegative. If $(a_{00}-\mu_0) = 0$, then $\mu_1 =\mu_2=\mu_3 = 0$.  If $(a_{00}-\mu_0) > 0$, then the minimality of $a_{00}$ implies $\mu_0 = 0$.
\end{proof}

\begin{lemma}\label{lemma2cases}
     If there is no $m \in \{0,1,2,3\}$ such that $a_{ii} = b_{m,ii}$ for at least two distinct indices $i \in \{0,1,2,3\} \setminus \{m\}$, then there exists an arrangement $\{i,j,k,l\} = \{0,1,2,3\}$ such that $a_{ii}d_i = a_{jj}d_j \neq a_{kk}d_k = a_{ll}d_l$, and $\varphi^{-1}(a_{ii}d_i) \cup \varphi^{-1}(a_{kk}d_k) =  \{(a_{00},0,0,0), (0,a_{11},0,0),(0,0,a_{22},0),(0,0,0,a_{33})\}$.
\end{lemma}

\begin{proof}
    Suppose that for every $m \in \{0,1,2,3\}$ we have $a_{ii} < b_{m,ii}$ for at least two distinct indices $i \in \{0,1,2,3\} \setminus \{m\}$ (note that $a_{ii} \leq b_{m,ii}$ always holds by definition). 
    
    Suppose that there exists $z \in \varphi^{-1}(a_{00} d_0)$ with $|\textup{supp}(z)| \geq 2$. By Lemma \ref{lemma_catenary_minrelations}, $0 \notin \textup{supp}(z)$. Without loss of generality, suppose that $z=(0,a_{01},a_{02},a_{03})$ where $a_{01},a_{02}>0$. Then, $a_{00} = b_{1,00} = b_{2,00}$. This implies $a_{11} < b_{2,11}$, $a_{22} < b_{1,22}$, $a_{33} < b_{1,33}$, and $a_{33} <b_{2,33}$. The last two inequalities give $\varphi^{-1}(a_{33}d_3) = \{(0,0,0,a_{33}),(a_{30},0,0,0)\}$, so $a_{30}>0$ and hence $a_{33} = b_{0,33}$. This implies $a_{11} < b_{0,11}$ and $a_{22} < b_{0,22}$. Now, since $a_{11} < b_{0,11}$ and $a_{11} < b_{2,11}$, we have $\varphi^{-1}(a_{11}d_1) = \{(0,a_{11},0,0),(0,0,0,a_{13})\}$, so $a_{13} > 0$ and hence $a_{11} = b_{3,11}$. Similarly, $\varphi^{-1}(a_{22}d_2) = \{(0,0,a_{22},0),(0,0,0,a_{23})\}$ and $a_{22} = b_{3,22}$, which is a contradiction, as now there exists at most one $i \in \{0,1,2\}$ such that $a_{ii} < b_{3,ii}$. 

    Hence, for every $z \in \varphi^{-1}(a_{00} d_0)$, we have $|\textup{supp}(z)| = 1$.
    Proceeding analogously with $a_{11}d_1$, $a_{22}d_2$, and $a_{33}d_3$, we obtain that for every $i \in \{0,1,2,3\}$ and every $z \in \varphi^{-1}(a_{ii} d_i)$, we have $|\textup{supp}(z)| = 1$.

    Without loss of generality, suppose that $a_{00}d_0 = a_{01}d_1$.
    Observe that $a_{01} \geq a_{11}$. If $a_{01} > a_{11}$, then
    \begin{equation}\label{notprop-eq1}
        a_{00}d_0 = a_{01}d_1 = ( a_{01}-a_{11})d_1 + a_{1i}d_i
    \end{equation}
    for some $i \in \{0,2,3\}$. Clearly $a_{1i} > 0$. If $i=0$, subtracting $d_0$ from both sides of \eqref{notprop-eq1} contradicts the minimality of $a_{00}$. If $i \neq 0$, we get a factorization $z \in  \varphi^{-1}(a_{00} d_0)$ with $|\textup{supp}(z)| = 2$, which is a contradiction. 
    
    Hence $a_{01} = a_{11}$, and therefore $a_{00}d_0 = a_{11}d_1$. Analogously, we get $a_{22}d_2 = a_{ii}d_i$ for some $i \in \{0,1,3\}$. If $i \in \{0,1\}$, then $a_{00}d_0 = a_{11}d_1 = a_{22}d_2$, so $a_{11} = b_{0,11}$ and $a_{22} = b_{0,22}$, which is contrary to our assumption. Hence $a_{22}d_2 = a_{33}d_3 \neq a_{00}d_0 = a_{11}d_1$, and the claim follows.
\end{proof}

If there exists $m \in \{0,1,2,3\}$ such that $a_{ii} = b_{m,ii}$ for at least two distinct indices $i \in \{0,1,2,3\} \setminus \{m\}$, we call $\langle d_0,d_1,d_2,d_3 \rangle$ \emph{primary}. Moreover, for a primary semigroup $\langle d_0,d_1,d_2,d_3 \rangle$, unless stated otherwise, we always assume throughout this paper that the generators are ordered so that $a_{11} = b_{0,11}$ and $a_{22} = b_{0,22}$. 

Numerical semigroups with embedding dimension four that are not primary will be referred to as \emph{secondary} semigroups. By Lemma \ref{lemma2cases}, if $\langle d_0,d_1,d_2,d_3 \rangle$ is a secondary semigroup, then there exists an arrangement $\{i,j,k,l\} = \{0,1,2,3\}$ such that $a_{ii}d_i = a_{jj}d_j \neq a_{kk}d_k = a_{ll}d_l$, and $\varphi^{-1}(a_{ii}d_i) \cup \varphi^{-1}(a_{kk}d_k) =  \{(a_{00},0,0,0), (0,a_{11},0,0), (0,0,a_{22},0),(0,0,0,a_{33})\}$. Moreover, for a secondary semigroup $\langle d_0,d_1,d_2,d_3 \rangle$, unless stated otherwise, we always assume throughout this paper that the generators are ordered so that $\varphi^{-1}(a_{00}d_0) = \varphi^{-1}(a_{11}d_1)= \{(a_{00},0,0,0), (0,a_{11},0,0)\}$ and $ \varphi^{-1}(a_{22}d_2) = \varphi^{-1}(a_{33}d_3) =  \{(0,0,a_{22},0), (0,0,0,a_{33})\}$. 

A numerical semigroup with embedding dimension four, $S=\langle d_0,d_1,d_2,d_3 \rangle$, is said to be \emph{ordered} if its generators are arranged according to the ordering defined above --- that is, the ordering assumed for primary semigroups if $S$ is primary, and the ordering assumed for secondary semigroups if $S$ is secondary. We adopt this convention so that certain definitions and constructions can be stated uniformly for any ordered semigroup $\langle d_0,d_1,d_2,d_3 \rangle$, without requiring a prior assumption on whether it is primary or secondary.

\subsection{Geometric framework}\label{subsection_framework}

From this point on, we will focus on the Ap{\'e}ry set of $\langle d_0,d_1,d_2,d_3 \rangle$ with respect to $d_0$, which we henceforth refer to simply as the Ap{\'e}ry set. For brevity, we also write $b_{0,ij} = b_{ij}$ for all $i,j \in \{0,1,2,3\}$.

Consider the unit cubes in $\mathbb{R}^3$. Each unit cube $[i,i+1] \times [j,j+1] \times [k,k+1] \subset \mathbb{R}^3$ has natural coordinates $(i,j,k) \in \mathbb{N}^3$ and is labeled with the value $id_1 + jd_2 + kd_3$. We will refer to the collection of all such labeled cubes as \emph{the initial collection of cubes}. We denote the unit cube $[i,i+1] \times [j,j+1] \times [k,k+1]$ by $[[i,j,k]]$.
For a point $P =(a,b,c) \in \mathbb{Z}^3$, we call the region $\{x > a, y>b, z>c\}$ \emph{the region associated to} $P$. We write $[[a,b,c]] \geq [[a',b',c']]$ when $a \geq a'$, $b \geq b'$, $c \geq c'$, and $[[a,b,c]] > [[a',b',c']]$ when $[[a,b,c]] \geq [[a',b',c']]$ and $[[a,b,c]] \neq [[a',b',c']]$.

\begin{lemma}\label{lemma_procedure_1}
   No cube labeled with an element of the Ap{\'e}ry set is in the region associated to a point $(\mu_1,\mu_2,\mu_3) \in \mathbb{Z}^3$ satisfying $\mu_1d_1 + \mu_2d_2 + \mu_3d_3 = \mu d_0$ for some positive integer $\mu$. 
\end{lemma}

\begin{proof}
    Recall that the Ap{\'e}ry set with respect to $d_0$ is the set of the smallest nonnegative integers from each residue class modulo $d_0$ that are in the numerical semigroup $\langle d_1,d_2,d_3\rangle$. Suppose to the contrary, that for some residue $r \in \{0,1,\dots,d_0-1\}$ there is a cube $[[\lambda_1,\lambda_2,\lambda_3]]$ labeled with the smallest number $\lambda d_0 + r$ ($\lambda \in \mathbb{N}$) that is in $\langle d_1,d_2,d_3\rangle$ such that $[[\lambda_1,\lambda_2,\lambda_3]] \geq [[\mu_1,\mu_2,\mu_3]]$. This gives
    \begin{equation}\label{lemma_procedure_1_eq}
     (\lambda_1 -\mu_1) d_1 + (\lambda_2-\mu_2) d_2 + (\lambda_3-\mu_3) d_3 = (\lambda-\mu) d_0 + r
    \end{equation}
    where all coefficients on the left-hand side are nonnegative. This means that the value in \eqref{lemma_procedure_1_eq} is congruent to $r$ modulo $d_0$, is nonnegative, is in $\langle d_1,d_2,d_3\rangle$, and is strictly less than $\lambda d_0 + r$, which contradicts the minimality of $\lambda d_0 + r$.
\end{proof}

\begin{thm}\label{theoremprocedure}
    Delete from the initial collection of cubes the regions $\{x > b_{11}\}$, $\{y > b_{22}\}$, and $\{z > b_{33}\}$.
    Next, for every point $(a_{01},a_{02},a_{03}) \in \mathbb{N}^3$ satisfying
     \begin{equation}\label{minrelation-0}
          a_{01} d_1 + a_{02}d_2 + a_{03} d_3 = a_{00} d_0
     \end{equation}
    delete the region associated to it. 
    Furthermore, for all identities of the form
     \begin{equation}\label{theoremprocedure_eq1}
         \lambda_0d_0 + \lambda_1d_1=\lambda_2d_2 + \lambda_3d_3
     \end{equation}
    with $\lambda_i \in \mathbb{Z}_{>0}$, $\lambda_2 < b_{22}$, and $\lambda_3 < b_{33}$, delete the region associated to $(-\lambda_1,\lambda_2,\lambda_3)$. For all identities of the form
    \begin{equation}\label{theoremprocedure_eq2}
         \lambda_0d_0 + \lambda_2d_2=\lambda_1d_1 + \lambda_3d_3
    \end{equation}
    with $\lambda_i \in \mathbb{Z}_{>0}$, $\lambda_1 < b_{11}$, and $\lambda_3 < b_{33}$, delete the region associated to $(\lambda_1,-\lambda_2,\lambda_3)$. Finally, for all identities of the form
    \begin{equation}\label{theoremprocedure_eq3}
        \lambda_0d_0 + \lambda_3d_3=\lambda_1d_1 + \lambda_2d_2
    \end{equation}
    with $\lambda_i \in \mathbb{Z}_{>0}$, $\lambda_1 < b_{11}$, and $\lambda_2 < b_{22}$, delete the region associated to $(\lambda_1,\lambda_2,-\lambda_3)$.
    Then, the remaining collection of cubes is the set of all cubes from the initial collection of cubes that are labeled with an element of the Ap{\'e}ry set.
\end{thm}

\begin{proof}
Let $R$ be the remaining collection of cubes.
The condition $\lambda_i < b_{ii}$ ensures that only the relevant identities are considered, since we delete the regions $\{x > b_{11}\}$, $\{y > b_{22}\}$, and $\{z > b_{33}\}$. Moreover, by Lemma \ref{lemma_procedure_1}, every cube labeled with an element of the Ap{\'e}ry set is in $R$, since deleting the regions $\{x > b_{11}\}$, $\{y > b_{22}\}$, and $\{z > b_{33}\}$ is the same as deleting the regions associated to $(b_{11}, - b_{12},-b_{13})$, $(-b_{21},b_{22},-b_{23})$, and $(-b_{31},-b_{32},b_{33})$. 

Therefore we only need to prove that $R$ contains only cubes labeled with an element of the Ap{\'e}ry set. Suppose to the contrary, that for some remainder $r \in \{0,1,\dots,d_0-1\}$, $\lambda d_0 + r$ and $\lambda' d_0 + r$, where $\lambda' > \lambda$, are labels in $R$. This means that
    \begin{equation}\label{eq1theorem}
        \lambda d_0 + r = \lambda_1d_1 + \lambda_2d_2 + \lambda_3d_3
    \end{equation}
     \begin{equation}\label{eq2theorem}
        \lambda' d_0 + r = \lambda_1'd_1 + \lambda_2'd_2 + \lambda_3'd_3
    \end{equation}
    where $[[\lambda_1,\lambda_2,\lambda_3]], [[\lambda_1',\lambda_2',\lambda_3']] \in R$.
    Subtracting \eqref{eq1theorem} from \eqref{eq2theorem} we get
    \begin{equation}\label{eq3theorem}
        (\lambda' - \lambda)d_0 = (\lambda_1' - \lambda_1)d_1 + (\lambda_2' - \lambda_2)d_2 + (\lambda_3' - \lambda_3)d_3
    \end{equation}
    where $(\lambda' - \lambda)>0$. Now, we consider cases of how many coefficients on the right-hand side of \eqref{eq3theorem} are positive. The left-hand side is positive; hence, at least one of the coefficients on the right-hand side is positive. 
    
If exactly one coefficient is positive, without loss of generality, suppose that $(\lambda_1' - \lambda_1)$ is the one. Moving the non-positive integers to the other side gives
    \begin{equation}\label{eq4theorem}
        (\lambda' - \lambda)d_0  + (\lambda_2 - \lambda_2')d_2 + (\lambda_3 - \lambda_3')d_3= (\lambda_1' - \lambda_1)d_1.
    \end{equation} 
Since all the above coefficients are nonnegative and $(\lambda_1' - \lambda_1)$ and $(\lambda' - \lambda) $ are positive, the minimality of $b_{11}$ gives $b_{11} \leq (\lambda_1' - \lambda_1) \leq \lambda_1'$, which contradicts $[[\lambda_1',\lambda_2',\lambda_3']] \in R$.

If exactly two coefficients on the right-hand side of \eqref{eq3theorem} are positive, without loss of generality, suppose that $(\lambda_1' - \lambda_1)$ and $(\lambda_2' - \lambda_2)$ are the two. Moving $(\lambda_3' - \lambda_3)$ to the other side gives
    \begin{equation}\label{eq5theorem}
        (\lambda' - \lambda)d_0  + (\lambda_3 - \lambda_3')d_3= (\lambda_1' - \lambda_1)d_1  + (\lambda_2' - \lambda_2)d_2.
    \end{equation}
The above identity satisfies the conditions of the current theorem, thus the region $\{x > (\lambda_1' - \lambda_1), y > (\lambda_2' - \lambda_2)\}$ was deleted. Again, this contradicts $[[\lambda_1',\lambda_2',\lambda_3']] \in R$, since $\lambda'_1 \geq (\lambda_1' - \lambda_1)$ and $\lambda'_2 \geq (\lambda_2' - \lambda_2)$.

Now, the only case left is where all the coefficients in \eqref{eq3theorem} are positive. We substitute $\mu =(\lambda' - \lambda)  > 0$ and $\mu_i =(\lambda_i' - \lambda_i)  >0$ for $i = 1,2,3$. We have
    \begin{equation}\label{eq6theorem}
         \mu d_0 = \mu_1d_1 + \mu_2d_2 + \mu_3d_3
    \end{equation}
and since $[[\mu_1,\mu_2,\mu_3]] \leq [[\lambda_1',\lambda_2',\lambda_3']] \in R$, we have $[[\mu_1,\mu_2,\mu_3]] \in R$. 
From the minimality of $a_{00}$, we get $a_{00} \leq \mu$. Subtracting $a_{00}d_0 = a_{01}d_1 + a_{02}d_2 + a_{03}d_3$ from \eqref{eq6theorem} gives
\begin{equation}\label{eq7theorem}
    (\mu - a_{00}) d_0 = (\mu_1 - a_{01})d_1 +  (\mu_2 - a_{02})d_2 + (\mu_3 - a_{03})d_3.
\end{equation}
    
If $\mu = a_{00}$, then $(0,\mu_1,\mu_2,\mu_3) \in \varphi^{-1}(a_{00}d_0)$, hence the region associated to $(\mu_1,\mu_2,\mu_3)$ had been deleted, which contradicts $[[\mu_1,\mu_2,\mu_3]] \in R$. 

Hence $(\mu - a_{00}) > 0$. Now, we repeat the same case by case consideration with \eqref{eq7theorem} as we did with \eqref{eq3theorem}, which results in all the coefficients in \eqref{eq7theorem} being positive. However, this implies that the cube $[[\mu_1,\mu_2,\mu_3]]$ is in the region associated to $(a_{01}, a_{02}, a_{03})$, which contradicts $[[\mu_1,\mu_2,\mu_3]] \in R$. 
\end{proof}

Throughout this paper, $R$ denotes the collection of all cubes from the initial collection of cubes that are labeled with an element of the Ap{\'e}ry set.
We will refer to equations of the forms \eqref{theoremprocedure_eq1}, \eqref{theoremprocedure_eq2}, \eqref{theoremprocedure_eq3} (where $\lambda_i \in \mathbb{Z}_{>0}$) as the \emph{$(2,2)$-type relations}. In addition, we will refer to points that are constructed from the $(2,2)$-type relations (as in Theorem \ref{theoremprocedure}) as the \emph{$(2,2)$-type points}. 

\subsection{$L$-shapes}\label{subsection_Lshapes}

Define an \emph{$L$-shape} to be a subset of the initial collection of cubes such that:
\begin{itemize}
    \item all its labels are elements of the Ap{\'e}ry set,
    \item every element of the Ap{\'e}ry set labels exactly one cube in it,
    \item if a cube $[[a,b,c]]$, $(a,b,c) \in \mathbb{N}^3$, is not in it, then neither is any cube in the region associated to $(a,b,c)$.
\end{itemize}
This notion is adopted from \cite{AGUILOGOST2015} and is established in the literature.

\begin{lemma}\label{lemma_procedure_2}
     Delete from $R$ one of the regions $\{x > a_{11}\}$, $\{y > a_{22}\}$, or $\{z > a_{33}\}$. Then every element of the Ap{\'e}ry set still labels some remaining cube.
\end{lemma}

\begin{proof}
Delete the region $\{x > a_{11}\}$ from $R$ (the other two cases follow analogously). 
If $a_{11} = b_{11}$, there is nothing to prove; thus, suppose $a_{11} < b_{11}$. Suppose to the contrary, that there is some element of the Ap{\'e}ry set that does not label any cube in $R \cap \{x \leq a_{11}\}$. By Theorem \ref{theoremprocedure}, then there exists $[[\lambda_1, \lambda_2,\lambda_3]] \in R$ with $\lambda_1 \geq a_{11}$ that is labeled with the element in question. There exists a positive integer $\lambda$ such that $a_{11} > (\lambda_1 - \lambda a_{11}) \geq 0$. Using $a_{11}d_1 = a_{12}d_2 + a_{13}d_3$, where $a_{12},a_{13} \in \mathbb{N}$ (recall $a_{11} < b_{11}$), we obtain
    \begin{equation}
        \lambda_1d_1 + \lambda_2d_2 + \lambda_3d_3 = (\lambda_1 - \lambda a_{11})d_1 + (\lambda_2+\lambda a_{12})d_2 + (\lambda_3+\lambda a_{13})d_3.
    \end{equation}
The cube $[[\lambda_1 - \lambda a_{11},\lambda_2+\lambda a_{12},\lambda_3+\lambda a_{13} ]]$ has the same label as $[[\lambda_1, \lambda_2,\lambda_3]]$, hence it is in $R$ by Theorem \ref{theoremprocedure}. Since $a_{11} >(\lambda_1 - \lambda a_{11})$, it is in $R \cap \{x \leq a_{11}\}$ --- contradiction.
\end{proof}

\begin{lemma}\label{lemma_procedure_3}
    Delete from $R$ the regions $\{x > a_{11}\}$, $\{y > a_{22}\}$, and $\{z > a_{33}\}$. Then, in the resulting collection, each element of the Ap{\'e}ry set labels at most one cube.
\end{lemma}
\begin{proof}
     Suppose to the contrary, that for two distinct $[[\lambda_1, \lambda_2, \lambda_3]], [[\lambda_1', \lambda_2', \lambda_3']] \in R \cap \{x \leq a_{11}, y \leq a_{22}, z \leq a_{33} \}$, we have
    \begin{equation*}
        \lambda_1d_1 + \lambda_2d_2 + \lambda_3d_3 = \lambda_1'd_1 + \lambda_2'd_2 + \lambda_3'd_3 
    \end{equation*}
    \begin{equation*}
        \implies  (\lambda_1' - \lambda_1)d_1 + (\lambda_2' - \lambda_2)d_2 + (\lambda_3' - \lambda_3)d_3 = 0.
    \end{equation*}
    If all coefficients above are zero, we have nothing to prove. Therefore either one or two coefficients above are positive. Regardless, without loss of generality, we can suppose that
    \begin{equation}\label{propeq3}
         (\lambda_1' - \lambda_1)d_1 = (\lambda_2 - \lambda_2')d_2 + (\lambda_3 - \lambda_3')d_3
    \end{equation}
    where every coefficient above is nonnegative. However, $(\lambda_1' - \lambda_1)$ is positive, because otherwise all coefficients are zero. This implies $a_{11} \leq (\lambda_1' - \lambda_1) \leq \lambda_1'$, which contradicts $[[\lambda_1', \lambda_2', \lambda_3']] \in R \cap \{x \leq a_{11}, y \leq a_{22}, z \leq a_{33} \}$.
\end{proof}

\begin{prop}\label{prop_primary}
    If $\langle d_0,d_1,d_2,d_3 \rangle$ is primary, then $R \cap \{z \leq a_{33}\}$ is an $L$-shape.
\end{prop}
\begin{proof}
    If $a_{ii} = b_{ii}$ for $i = 1,2,3$, then $R = R \cap \{x \leq a_{11}, y \leq a_{22}, z \leq a_{33} \}$, hence Lemma \ref{lemma_procedure_3} yields the claim. Otherwise, we have $a_{ii} < b_{ii}$ for exactly one $i \in \{1,2,3\}$. In such case, Lemmas \ref{lemma_procedure_2} and \ref{lemma_procedure_3} yield the claim.
\end{proof}

\begin{prop}\label{prop_secondary}
    Suppose that for $\langle d_0,d_1,d_2,d_3 \rangle$ we have $a_{11}=b_{11}$ and $a_{22}d_2 = a_{33}d_3$. Then there is a bijection $f : R \cap \{y \geq a_{22}\} \to R \cap \{z \geq a_{33}\}$ sending $[[\lambda_1,\lambda_2,\lambda_3]] \mapsto [[\lambda_1,\lambda_2-a_{22},\lambda_3+a_{33}]]$. Moreover, $R \cap \{y \leq a_{22}\}$ and $R \cap \{z \leq a_{33}\}$ are $L$-shapes.
\end{prop}

\begin{proof}
    For $[[\lambda_1,\lambda_2,\lambda_3]] \in R \cap \{y \geq a_{22}\} $, the cube $[[\mu_1,\mu_2 - a_{22},\mu_3 + a_{33}]]$ is in $R$ by Theorem \ref{theoremprocedure}, because it has the same label as $[[\lambda_1,\lambda_2,\lambda_3]]$, so it is labeled with an element of the Ap{\'e}ry set. Thus, the function $f: [[\lambda_1,\lambda_2,\lambda_3]] \mapsto [[\lambda_1,\lambda_2-a_{22},\lambda_3+a_{33}]]$ maps $R \cap \{y \geq a_{22}\}$ to $R \cap \{z \geq a_{33}\}$ and is also clearly injective. Similarly, we obtain the injection $f^{-1} : R \cap \{z \geq a_{33}\} \to R \cap \{y \geq a_{22}\} $ sending $[[\lambda_1,\lambda_2,\lambda_3]] \mapsto [[\lambda_1,\lambda_2+a_{22},\lambda_3-a_{33}]]$, which yields that $f$ is a bijection.   
    
    Let $T = R \cap \{z \leq a_{33}\}$. Since $f$ preserves labels, every element of the Ap{\'e}ry set labels a cube in $T$. Thus, to prove that $T$ is an $L$-shape we only have to prove that it has distinct labels.  
     
    Otherwise, we have
    \begin{equation}\label{propnottrick-eq1}
        \mu_1d_1 + \mu_2d_2 + \mu_3d_3 = \nu_1d_1 + \nu_2d_2 + \nu_3d_3
    \end{equation}
    where $[[\mu_1,\mu_2,\mu_3]]$ and $[[\nu_1,\nu_2,\nu_3]]$ are two distinct cubes in $T$. Without loss of generality, suppose $\mu_2 \geq \nu_2$. 
    
    If $\mu_1 \leq \nu_1$ and $\mu_3 \leq \nu_3$, then
    \begin{equation}\label{propnottrick-eq2}
         (\mu_2-\nu_2)d_2  = (\nu_1-\mu_1)d_1  + (\nu_3-\mu_3)d_3
    \end{equation}
    where all coefficients above are nonnegative, which implies $(\mu_2-\nu_2) \geq a_{22}$ (as $[[\mu_1,\mu_2,\mu_3]] \neq [[\nu_1,\nu_2,\nu_3]]$). If $\nu_3 \geq (\nu_3-\mu_3) \geq a_{33}$, we get a contradiction with $[[\nu_1,\nu_2,\nu_3]] \in T$, hence $(\nu_3-\mu_3) < a_{33}$. This gives
    \begin{equation}\label{propnottrick-eq21}
         (\mu_2-\nu_2-a_{22})d_2 + (a_{33}-\nu_3+\mu_3)d_3  = (\nu_1-\mu_1)d_1
    \end{equation}
    where all coefficients above are nonnegative and $(a_{33}-\nu_3+\mu_3)$ is positive. Hence $(\nu_1-\mu_1)$ is also positive, which implies $\nu_1 \geq (\nu_1-\mu_1) \geq a_{11} = b_{11}$, contradicting $[[\nu_1,\nu_2,\nu_3]] \in T$.

    If $\mu_1 > \nu_1$ and $\mu_3 \leq \nu_3$, then
    \begin{equation}\label{propnottrick-eq3}
         (\mu_2-\nu_2)d_2 + (\mu_1-\nu_1)d_1  =  (\nu_3-\mu_3)d_3
    \end{equation}
    where all coefficients above are nonnegative and $(\mu_1-\nu_1)$ is positive, which implies $\nu_3 \geq (\nu_3-\mu_3) \geq a_{33}$, contradicting $[[\nu_1,\nu_2,\nu_3]] \in T$.

    Lastly, if $\mu_1 \leq \nu_1$ and $\mu_3 > \nu_3$, then
    \begin{equation}\label{propnottrick-eq4}
         (\mu_2-\nu_2)d_2 + (\mu_3-\nu_3)d_3  =  (\nu_1-\mu_1)d_1
    \end{equation}
     where all coefficients above are nonnegative and $(\mu_3-\nu_3)$ is positive, which implies $\nu_1 \geq (\nu_1-\mu_1) \geq a_{11} = b_{11}$, so again a contradiction with $[[\nu_1,\nu_2,\nu_3]] \in T$.

    Analogously, we obtain that $R \cap \{y \leq a_{22}\}$ is an $L$-shape.
\end{proof}

In particular, since secondary semigroups satisfy $a_{11}=b_{11}$ and $a_{22}d_2 = a_{33}d_3$ by definition, Proposition \ref{prop_secondary} applies to all secondary semigroups.

\begin{ex}\label{example_T}
Consider $\langle d_0,d_1,d_2,d_3 \rangle = \langle 196,225,256,289 \rangle$, which has embedding dimension four. One can check that
\begin{alignat*}{2}
    &a_{00}d_0 = 17 \cdot 196 &&= 2 \cdot 225 + 9 \cdot 256 + 2 \cdot 289,\\
    &a_{11}d_1 = 7 \cdot 225 &&= 1 \cdot 196 + 2 \cdot 256 +3 \cdot 289,\\
   &a_{22}d_2 =15 \cdot 256 &&= 15 \cdot 196 + 4 \cdot 225,\\
   &a_{33}d_3 =4 \cdot 289 &&= 4 \cdot 225 + 1 \cdot 256.
\end{alignat*}
Hence $\langle 196,225,256,289 \rangle$ is primary. Figure \ref{fig_theorem_primary} shows the corresponding collection $R$. The cubes colored red are in the region $R \cap \{z \geq a_{33}\}$ and the cubes colored gray form the $L$-shape from Proposition \ref{prop_primary}. 
Next, consider $\langle d_0,d_1,d_2,d_3 \rangle =\langle 145,203,74,111 \rangle$, which has embedding dimension four. One can check that
\begin{alignat*}{2}
    & \varphi^{-1}(a_{00}d_0) = \varphi^{-1}(7 \cdot 145) &&= \{(7,0,0,0),(0,5,0,0)\},\\
    &\varphi^{-1}(a_{11}d_1) =\varphi^{-1}(5 \cdot 203) &&= \{(7,0,0,0),(0,5,0,0)\},\\
   &\varphi^{-1}(a_{22}d_2) = \varphi^{-1}( 3\cdot 74) &&= \{(0,0,3,0),(0,0,0,2)\},\\
   &\varphi^{-1}(a_{33}d_3) = \varphi^{-1}(2 \cdot 111)  &&= \{(0,0,3,0),(0,0,0,2)\}.
\end{alignat*}
Hence $\langle 145,203,74,111 \rangle$ is secondary.  Figure \ref{fig_theorem_secondary} shows the corresponding collection $R$. Similarly, red marks the region $R \cap \{z \geq a_{33}\}$, while gray marks an $L$-shape from Proposition \ref{prop_secondary}. 
\end{ex}

\begin{figure}[h]
    \centering
    \begin{minipage}[t]{0.48\textwidth}
        \centering
        \includegraphics[width=0.7\textwidth]{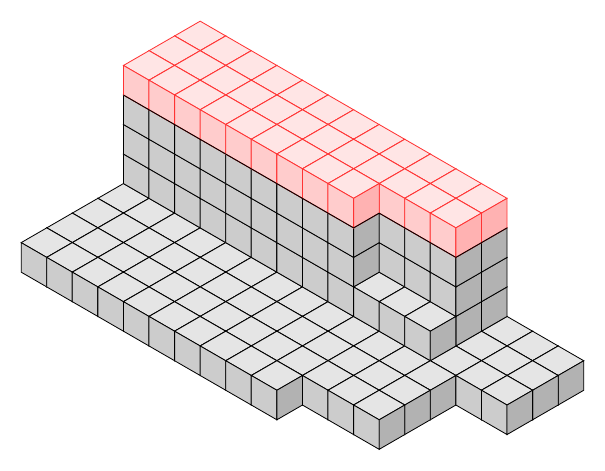}
        \caption{}
        \label{fig_theorem_primary}
    \end{minipage}
    \begin{minipage}[t]{0.48\textwidth}
        \centering
        \includegraphics[width=0.7\textwidth]{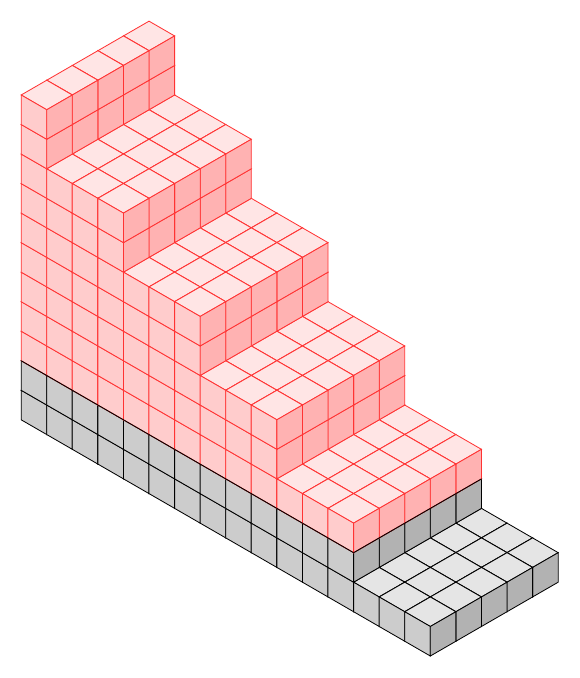}
        \caption{}
        \label{fig_theorem_secondary}
    \end{minipage}
\end{figure}

\section{Minimal presentations}\label{section_minimal_presentations}

In this section, we investigate the set of Betti elements for primary and secondary semigroups. We start with preliminary definitions and then proceed to the main results. Theorems \ref{theorembettifind_primary} and \ref{theorembettifind_secondary} relate Betti elements to an $L$-shape for primary and secondary semigroups respectively. Lemmas \ref{lemma_catenary_primary} and \ref{lemma_catenarycase1_primary} determine the sets of factorizations of Betti elements arising from $(2,2)$-type points in the primary case, while Lemmas \ref{lemma_catenary_lambda1_secondary} and \ref{lemma_catenary_lambda23_secondary} establish the analogous results for secondary semigroups. Finally, Lemma \ref{lemma_catenaryT_primary} characterizes the $(2,2)$-type points which do not correspond to Betti elements in the primary case, while Lemmas \ref{lemma_catenaryT_secondary_case1} and \ref{lemma_catenaryT_secondary_case2} do the same in the secondary case. This framework will be used in Section \ref{section_type} to study minimal presentations from a geometric perspective.

Bresinsky’s work \cite{Bresinsky1988}, which underlies the bound in Theorem \ref{theoremtypebre}, contains results closely related to those of this section (cf.\ \cite[Theorem 3]{Bresinsky1988}). Our approach differs in two main respects. First, in the classification of numerical semigroups with $e(S)=4$, we distinguish primary and secondary semigroups, whereas Bresinsky introduced three families. Second, we relate Betti elements and minimal presentations to the Ap{\'e}ry set via an $L$-shape, allowing us to make use of its geometric structure, whereas Bresinsky's approach was commutative algebra centered.

Let $S$ be minimally generated by $\{d_0,d_1,\dots,d_k\}$.
Two elements $z$ and $z'$ of $\mathbb{N}^{k+1}$ are said to be \emph{$\mathcal{R}$-related} if there exists a finite sequence $z = z_1, \dots, z_p = z'$ in $\mathbb{N}^{k+1}$ such that $\text{supp}(z_i) \cap \text{supp}(z_{i+1})$ is nonempty for all $i \in \{1, \dots, p-1\}$. We write $z \, \mathcal{R} \, z'$ in this case. We refer to the equivalence classes of $\varphi^{-1}(s)$ under the relation $\mathcal{R}$ as the \emph{$\mathcal{R}$-classes}. 
An element $s \in S$ is called a \emph{Betti element} if there exist two elements of $\varphi^{-1}(s)$ that are not $\mathcal{R}$-related (i.e.\ $\varphi^{-1}(s)$ has at least two $\mathcal{R}$-classes). The set of all Betti elements of $S$ is denoted $\text{Betti}(S)$.

Let $\mathcal{R}_1, \dots, \mathcal{R}_l$ be all the $\mathcal{R}$-classes of $\varphi^{-1}(b)$ for some $b \in \text{Betti}(S)$. We choose $v_i \in \mathcal{R}_i$ for each $i \in \{1,\dots,l\}$. Let $\rho_b$ be a set of $l-1$ pairs of elements from $\{v_1, \dots, v_l\}$ such that the graph with vertices $\{v_1, \dots, v_l\}$, having an edge between $v_i$ and $v_j$ if and only if $(v_i,v_j)$ or $(v_j,v_i)$ is in $\rho_b$, is connected. 
\begin{thm}\label{theorem_minimalpresentation}
    \textup{\cite{Rosales1996MinimalRelation}}
    The set $\rho = \bigcup_{b \in \text{Betti}(S)} \rho_b$ is a minimal presentation of $S$ and every minimal presentation of $S$ is of this form.
\end{thm}

For the rest of this section, let $S = \langle d_0,d_1,d_2,d_3 \rangle$ be an ordered numerical semigroup with embedding dimension four. Let $T = R \cap \{z \leq a_{33}\}$. By Propositions \ref{prop_primary} and \ref{prop_secondary}, $T$ is an $L$-shape.
Let $V \subset \mathbb{Z}^3$ be a minimal set (with respect to inclusion) of points $(x,y,z)$ such that $xd_1 + yd_2 +zd_3$ is a nonnegative multiple of $d_0$, and such that removing from the initial collection of cubes the regions associated to the points of $V$ yields $T$. 
Define
$$ U = \{ \max\{ x,0\} d_1 + \max\{y,0\}d_2 + \max\{z,0\}d_3 \mid (x,y,z) \in V \}.$$
Note that $U$ is independent of the choice of $V$. Furthermore, assume that every point $(x,y,z) \in V$ such that
$$ \max\{ x,0\} d_1 + \max\{y,0\}d_2 + \max\{z,0\}d_3 \notin \{(a_{00}d_0,b_{11}d_1,b_{22}d_2,a_{33}d_3)\}$$
is a $(2,2)$-type point. Such $V$ exists by Theorem \ref{theoremprocedure}. Throughout this paper, we will always assume that $V$ satisfies this.

\begin{remark}
   The above assumption may fail for certain choices of $V$. For example, if $a_{00}d_0 = a_{11}d_1$ and there exists $(-\lambda_1,\lambda_2,\lambda_3) \in V$ such that $\lambda_1 = \mu a_{11}$ for some $\mu \in \mathbb{Z}_{>0}$, then we could take $(0,\lambda_2,\lambda_3)$ as an element of $V$ instead of $(-\lambda_1,\lambda_2,\lambda_3)$, since $\lambda_2d_2 + \lambda_3d_3 = (\lambda_0 + \mu a_{00})d_0$.
\end{remark}

For the semigroups $\langle 196,225,256,289 \rangle$ and $\langle 145,203,74,111 \rangle$ from Example \ref{example_T}, shown respectively in Figures \ref{fig_betti_primary} and \ref{fig_betti_secondary}, the cubes shown in gray form the corresponding collections $T$. The corresponding sets $U$ are the set of labels of the cubes shown in red.

\begin{figure}[h]
    \centering
    \begin{minipage}[t]{0.48\textwidth}
        \centering
        \includegraphics[width=0.7\textwidth]{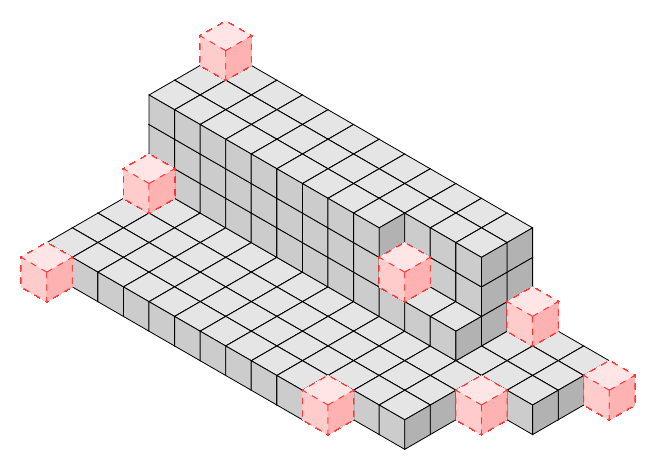}
        \caption{}
        \label{fig_betti_primary}
    \end{minipage}
    \begin{minipage}[t]{0.48\textwidth}
        \centering
        \includegraphics[width=0.7\textwidth]{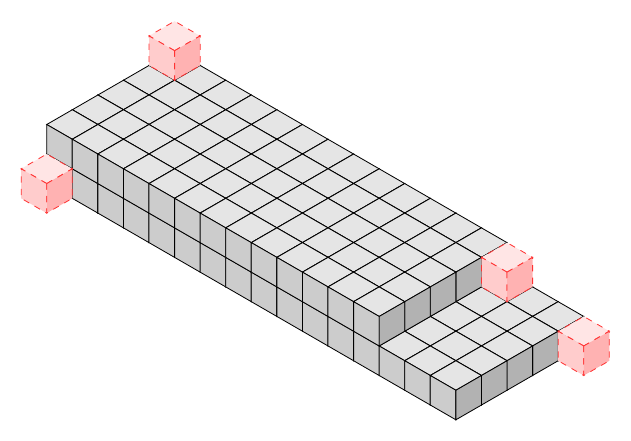}
        \caption{}
        \label{fig_betti_secondary}
    \end{minipage}
\end{figure}

\subsection{Primary case}

\begin{lemma}\label{lemmabettifind_primary}
  Suppose that $S$ is primary. If $b \in \textup{Betti}(S)$ labels a cube in $R$, then $b \in U$.
\end{lemma}
    
\begin{proof}
    In this proof, it will be more convenient to drop the assumption that $S$ is ordered --- we only suppose that $a_{ii}=b_{ii}$ for at least two distinct $i \in \{1,2,3\}$ and that  $T = R \cap \{x \leq a_{11},y\leq a_{22}, z \leq a_{33}\}$.
    
    Suppose that $b$ labels a cube in $R$. Then $b \in \textup{Ap}(S,d_0)$ by Theorem \ref{theoremprocedure}. Since $T$ is an $L$-shape, $b$ labels a cube in $T$.   Consider the factorization $z_0 =(0,\lambda_1,\lambda_2, \lambda_3) \in \varphi^{-1}(b)$ with $[[\lambda_1,\lambda_2,\lambda_3]] \in T$. Let $(\lambda_0',\lambda_1',\lambda_2', \lambda_3')$ be a different factorization of $b$. We have
    \begin{equation}\label{betticharacter_eq0}
        (\lambda_1-\lambda_1')d_1 +  (\lambda_2-\lambda_2')d_2 +  (\lambda_3-\lambda_3')d_3 = \lambda_0'd_0.
    \end{equation}

    If all coefficients on the left-hand side are nonnegative, then $\lambda_0'>0$, as the two factorizations are distinct. This contradicts $[[\lambda_1,\lambda_2,\lambda_3]] \in T$ by Lemma \ref{lemma_procedure_1}.

    Suppose that exactly one coefficient on the left-hand side of \eqref{betticharacter_eq0} is nonnegative. Without loss of generality, suppose that every coefficient below is nonnegative.
    \begin{equation}\label{betticharacter_eq00}
        (\lambda_1-\lambda_1')d_1   = \lambda_0'd_0 + (\lambda_2'-\lambda_2)d_2 +  (\lambda_3'-\lambda_3)d_3.
    \end{equation}
    This gives $\lambda_1 \geq (\lambda_1-\lambda_1') \geq a_{11}$, which contradicts $[[\lambda_1,\lambda_2,\lambda_3]] \in T$.

    Suppose that exactly two coefficients on the left-hand side of \eqref{betticharacter_eq0} are nonnegative. Without loss of generality, suppose that every coefficient below is nonnegative.
     \begin{equation}\label{betticharacter_eq01}
        (\lambda_1-\lambda_1')d_1 +  (\lambda_2-\lambda_2')d_2  = \lambda_0'd_0 +  (\lambda_3'-\lambda_3)d_3.
    \end{equation}
    If $\lambda_0'>0$, then by Lemma \ref{lemma_procedure_1}, we get a contradiction with $[[\lambda_1,\lambda_2,\lambda_3]] \in T$, hence $\lambda_0'=0$. This means that for all $z \in \varphi^{-1}(b)$ we have $0 \notin \text{supp}(z)$, as $(\lambda_0',\lambda_1',\lambda_2', \lambda_3') \in \varphi^{-1}(b)$ was chosen arbitrarily.

    Let $z=(0,\mu_1, \mu_2,  \mu_3) \in \varphi^{-1}(b)$ be a factorization that is not $\mathcal{R}$-related to $z_0$. Repeating the same argument as for $(\lambda_0',\lambda_1',\lambda_2', \lambda_3')$, without loss of generality, we may assume $\lambda_1 \geq \mu_1$, $\lambda_2 \geq \mu_2$, $\lambda_3 < \mu_3$. This forces $\mu_1= \mu_2= \lambda_3 =0$, as otherwise $z \, \mathcal{R} \, z_0$. Hence
    \begin{equation}\label{betticharacter_eq03}
       b= \lambda_1d_1 + \lambda_2d_2   =  \mu_3d_3
    \end{equation}
    where $\lambda_1, \lambda_2, \mu_3> 0$ --- if $\lambda_1 =0$ or $\lambda_2 = 0$, we derive a contradiction as in \eqref{betticharacter_eq00}.
    If $ \mu_3 > a_{33}$, then 
    \begin{equation}
        b  =  a_{30}d_0 + a_{31}d_1 + a_{32}d_2 +  (\mu_3-a_{33})d_3
    \end{equation}
    where $(a_{30},a_{31},a_{32},0) \in \varphi^{-1}(a_{33}d_3)$. We also have $a_{30} = 0$, since $0 \notin \text{supp}(z)$ for any $z \in \varphi^{-1}(b)$. Hence $\max \{a_{31},a_{32}\} > 0$, which implies $z_0 \, \mathcal{R} \, (0,a_{31},a_{32},\mu_3-a_{33}) \, \mathcal{R} \,z$, which is a contradiction. Thus $\mu_3 = a_{33}$, and hence $b = a_{33}d_3 \in U$.
\end{proof}

\begin{thm}\label{theorembettifind_primary}
If $S$ is primary, then $\textup{Betti}(S) \subseteq U$. 
\end{thm}

\begin{proof}
    Suppose that there exists an element $b \in \text{Betti}(S)$ that is not in $U$.
    There exists a factorization $z_0 =(0,\kappa_1,\kappa_2,\kappa_3) \in \varphi^{-1}(b)$; otherwise, if every factorization $z \in \varphi^{-1}(b)$ had $0 \in \text{supp}(z)$, then all factorizations would be $\mathcal{R}$-related, which cannot happen for a Betti element. By Lemma \ref{lemmabettifind_primary}, we have $[[\kappa_1,\kappa_2,\kappa_3]] \notin T$, hence $[[\kappa_1,\kappa_2,\kappa_3]]$ is in the region associated to some point $P \in V$.

    \begin{itemize}

    \item  If $P =(-\lambda_1,\lambda_2,\lambda_3)$, where
    \begin{equation}\label{betticharacter_eq1}
        \lambda_0d_0 + \lambda_1d_1 = \lambda_2d_2 + \lambda_3d_3
    \end{equation}
    with $\lambda_i \in \mathbb{Z}_{>0}$, we get
     \begin{equation}
        b =\lambda_0d_0 + (\lambda_1+\delta_1)d_1 + \delta_2d_2 + \delta_3d_3  =  \delta_1d_1 + (\lambda_2+\delta_2)d_2 + (\lambda_3+\delta_3)d_3
    \end{equation}
    for some $\delta_i \in \mathbb{N}$ with $\max \{\delta_1,\delta_2,\delta_3\} > 0$. The factorizations $z_1 =(\lambda_0,\lambda_1+\delta_1,\delta_2,\delta_3)$ and $z_2 =(0, \delta_1,\lambda_2+\delta_2,\lambda_3+\delta_3)$ have a nonempty intersection of supports, and hence they are $\mathcal{R}$-related. Moreover, we have $0,1,2,3 \in \textup{supp}(z_1) \cup \textup{supp}(z_2)$, so every factorization of $b$ will be $\mathcal{R}$-related to $z_1$ and $z_2$. This contradicts $b \in \textup{Betti}(S)$, as it must have at least two distinct $\mathcal{R}$-classes. 
    
    The cases $P=(\lambda_1,-\lambda_2,\lambda_3)$ and $P=(\lambda_1,\lambda_2,-\lambda_3)$ follow analogously.
    
    \item If $P =(a_{01},a_{02},a_{03})$, where $(0,a_{01},a_{02},a_{03}) \in \varphi^{-1}(a_{00}d_0)$, we get
     \begin{equation}
        b =a_{00}d_0 + \delta_1d_1 + \delta_2d_2 + \delta_3d_3  =  (a_{01}+\delta_1)d_1 + (a_{02}+\delta_2)d_2 + (a_{03}+\delta_3)d_3
    \end{equation}
    for some $\delta_i \in \mathbb{N}$ with $\max \{\delta_1,\delta_2,\delta_3\} > 0$. The factorizations $z_1 =(a_{00},\delta_1,\delta_2,\delta_3)$ and $z_2 =(0,a_{01}+\delta_1,a_{02}+\delta_2,a_{03}+\delta_3)$ have a nonempty intersection of supports, and hence they are $\mathcal{R}$-related. Because $b \in \textup{Betti}(S)$, there exists a factorization $z \in \varphi^{-1}(b)$ which is not $\mathcal{R}$-related to them. 
    
    In this case, as in the proof of Lemma \ref{lemmabettifind_primary}, it will be more convenient to drop the assumption that $S$ is ordered --- we only suppose that $a_{ii}=b_{ii}$ for at least two distinct $i \in \{1,2,3\}$ and that  $T = R \cap \{x \leq a_{11},y\leq a_{22}, z \leq a_{33}\}$. This way, without loss of generality, we can suppose that $\delta_1 > 0$. As $a_{00} >0$, we have $0,1 \notin \text{supp}(z)$. 

    First, suppose that $z = (0,0,\mu_2,\mu_3)$ where $\mu_2,\mu_3>0$. This implies $z_1 = (a_{00},\delta_1,0,0 )$ and $z_2 =(0,a_{01}+\delta_1,0,0)$. 
    If $\mu_2 \geq a_{22}$, then
   \begin{equation}\label{betticharacter_eq40}
         b =\mu_2d_2 + \mu_3d_3 = a_{20}d_0 + a_{21}d_1 +(\mu_2-a_{22})d_2 + (\mu_3+a_{23})d_3
   \end{equation}
   where $(\mu_2-a_{22}) \geq 0$ and $(a_{20},a_{21},0,a_{23}) \in \varphi^{-1}(a_{22}d_2)$.
   If $\max \{a_{20},a_{21} \}>0$, then $z_1 \, \mathcal{R} \,(a_{20},a_{21},$ $ \mu_2-a_{22},\mu_3+a_{23}) \, \mathcal{R} \, z$ (recall $\mu_3>0$), which is a contradiction. Hence $a_{20}=a_{21}=0$ and $a_{22}d_2 = a_{23}d_3$, which implies $a_{23} \geq a_{33}$. Moreover, since $(a_{20},a_{21},0,a_{23})$ was chosen arbitrarily, the conclusion $a_{20}=0$ gives $a_{22} < b_{22}$. Now, we have
   \begin{equation}\label{betticharacter_eq401}
        b=  a_{30}d_0 + a_{31}d_1 + (\mu_2-a_{22} + a_{32})d_2 + (\mu_3+a_{23}-a_{33})d_3
   \end{equation}
   where all coefficients above are nonnegative (since $a_{23} \geq a_{33}$) and $(a_{30},a_{31},a_{32},0) \in \varphi^{-1}(a_{33}d_3)$. As in \eqref{betticharacter_eq40}, this yields $a_{30} = 0$ and hence $a_{33} < b_{33}$. Since we also have $a_{22} < b_{22}$, we get a contradiction with the hypothesis.
   
   If $\mu_3 \geq a_{33}$ we proceed analogously and get a contradiction, thus $\mu_2 < a_{22}$ and $ \mu_3 < a_{33}$. 
   Consider the cube $[[0,\mu_2,\mu_3]]$. If it is in $T$, then $b \in U$ by Lemma \ref{lemmabettifind_primary}. Hence $[[0,\mu_2,\mu_3]]$ is in the region associated to a point $Q \in V$.
   As $\mu_2 < a_{22}$ and $ \mu_3 < a_{33}$, we have $Q =(-\nu_1,\nu_2,\nu_3)$, where
    \begin{equation}\label{betticharacter_eq41}
        - \nu_1d_1 + \nu_2d_2 +  \nu_3d_3   =  \nu_0d_0
    \end{equation}
    $\nu_0,\nu_2,\nu_3 \in \mathbb{Z}_{>0}$, $\nu_1 \in \mathbb{N}$ (possibly $\nu_1=0$), and $\mu_2 \geq \nu_2$, $\mu_3 \geq \nu_3$.  We can write 
    \begin{equation}\label{betticharacter_eq42}
        b =  \nu_0d_0 + \nu_1d_1 + (\mu_2-\nu_2)d_2 + (\mu_3-\nu_3)d_3.
    \end{equation}
    We have $\max \{\mu_2-\nu_2,\mu_3-\nu_3 \}>0$, as $b \notin U$, hence $z \, \mathcal{R} \, (\nu_0,\nu_1,\mu_2-\nu_2,\mu_3-\nu_3) \, \mathcal{R}  \, z_1$, which is a contradiction.

    Next, suppose that $z = (0,0,\mu_2,0)$. 
    If $\mu_2 = a_{22} $, then $b = a_{22}d_2 \in U$, so $\mu_2 > a_{22}$. We can write
    \begin{equation}\label{betticharacter_eq6}
         b=\mu_2d_2 = a_{20}d_0 + a_{21}d_1 + (\mu_2-a_{22})d_2 + a_{23}d_3.
    \end{equation}
    where $(a_{20},a_{21},0,a_{23}) \in \varphi^{-1}(a_{22}d_2)$.
    Because $(\mu_2-a_{22}) > 0$, we have $z \, \mathcal{R} \, (a_{20},a_{21},\mu_2-a_{22},a_{23})$. Since $z$ is not $\mathcal{R}$-related to $z_1$, we have $a_{20} = a_{21} = 0$ (recall $a_{00},\delta_1 >0$). This gives the factorization $(0,0,\mu_2-a_{22},a_{23}) \in \varphi^{-1}(b)$ that has $(\mu_2-a_{22}),a_{23} > 0$ and which is not $\mathcal{R}$-related to $z_1$ nor $z_2$, as it is $\mathcal{R}$-related to $z$. Thus we reduce to the previously treated case $z =(0,0,\mu_2,\mu_3)$. 

    The case $z = (0,0,0,\mu_3)$ follows analogously.

    \item   If $P =(b_{11},-b_{12},-b_{13})$, where $(b_{10},0,b_{12},b_{13}) \in \varphi^{-1}(b_{11}d_1)$ with $b_{10}>0$, we get
    \begin{equation}
        b = (b_{11}+\delta_1)d_1 +   \delta_2d_2 + \delta_3d_3  =  b_{10}d_0 + \delta_1d_1+ (b_{12}+\delta_2)d_2 + (b_{13}+\delta_3)d_3
    \end{equation}
    for some $\delta_i \in \mathbb{N}$ with $\max \{\delta_1,\delta_2,\delta_3\} > 0$. The factorizations $z_1 =(0,b_{11} +\delta_1,\delta_2,\delta_3)$ and $z_2 =(b_{10},\delta_1,b_{12}+\delta_2,b_{13}+\delta_3)$ have a nonempty intersection of supports, and hence they are $\mathcal{R}$-related. Because $b \in \text{Betti}(S)$, there exists a factorization $z \in \varphi^{-1}(b)$ that is not $\mathcal{R}$-related to them. 
        
        If $b_{12},b_{13}>0$, then $0,1,2,3 \in \textup{supp}(z_1) \cup \textup{supp}(z_2)$, which implies that $b$ has only one $\mathcal{R}$-class.
        
        If $b_{12}>0,b_{13} = 0$, then $0,1,2 \in \textup{supp}(z_1) \cup \textup{supp}(z_2)$ and hence $z = (0,0,0,\mu_3)$. 
        If $\mu_3 = a_{33}$, then $b = a_{33}d_3 \in U$, hence $\mu_3 > a_{33}$. Now, for $(a_{30},a_{31},a_{32},0) \in \varphi^{-1}(a_{33}d_3)$, we have $z \, \mathcal{R} \, (a_{30},a_{31},a_{32},\mu_3-a_{33}) \, \mathcal{R} \, z_1  \, \mathcal{R} \, z_2$, since $\max\{a_{30},a_{31},a_{32}\}>0$, which is a contradiction.

        If $b_{12}=0, b_{13} > 0$, we proceed analogously as above.

        If $b_{12} = b_{13} = 0$, then $P = (b_{11},0,0)$ and $b_{10} \geq a_{00}$. If $b_{10} = a_{00}$, then we are in the case $P = (a_{01},a_{02},a_{03})$ and we can proceed analogously as there. If $b_{10} > a_{00}$, then $(b_{10}-a_{00},a_{01},a_{02},a_{03}) \in \varphi^{-1}(b_{11}d_1)$ with $\max \{a_{01},a_{02},a_{03}\}>0$. If $a_{01} > 0$, we get a contradiction with the minimality of $b_{11}$ (since $(b_{10}-a_{00})>0$). Hence $(b_{10}-a_{00},0,a_{02},a_{03}) \in \varphi^{-1}(b_{11}d_1)$ where $(b_{10}-a_{00})>0$ and $\max \{a_{02},a_{03}\} > 0$. Therefore we can replace $P = (b_{11},0,0)$ with $(b_{11},-a_{02},-a_{03})$ and proceed as in the previous cases of this bullet-point.

        The case $P = (-b_{21},b_{22},-b_{23})$ follows analogously.

    \item   Suppose $P =(-a_{31},-a_{32},a_{33})$, where $(a_{30},a_{31},a_{32},0) \in \varphi^{-1}(a_{33}d_3)$. If $a_{33}=b_{33}$, then we can proceed as in the previous bullet-point. Thus, suppose $a_{33}<b_{33}$, which implies $a_{30} = 0$. We get
    \begin{equation}
        b = \delta_1d_1 +   \delta_2d_2 + (a_{33} +\delta_3)d_3  =  (a_{31}+\delta_1)d_1+ (a_{32}+\delta_2)d_2 + \delta_3d_3
    \end{equation}
    for some $\delta_i \in \mathbb{N}$ with $\max \{\delta_1,\delta_2,\delta_3\} > 0$. The factorizations $z_1 =(0,\delta_1,\delta_2,a_{33} +\delta_3)$ and $z_2 =(0,a_{31}+\delta_1,a_{32}+\delta_2,\delta_3)$ have a nonempty intersection of supports, and hence they are $\mathcal{R}$-related. Because $b \in \text{Betti}(S)$, there exists a factorization $z \in \varphi^{-1}(b)$ that is not $\mathcal{R}$-related to them. 
        
    If $a_{31},a_{32}>0$, then $1,2,3 \in \textup{supp}(z_1) \cup \textup{supp}(z_2)$, hence $z = (\mu_0,0,0,0)$. Consider the cube $[[a_{31}+\delta_1,a_{32}+\delta_2,\delta_3]]$ which has label $b$. By Lemma \ref{lemma_catenary_help_primary}, we have $[[a_{31}+\delta_1,a_{32}+\delta_2,\delta_3]] \notin T$. If $\delta_3 < a_{33}$, then $[[a_{31}+\delta_1,a_{32}+\delta_2,\delta_3]]$ lies in the region associated to some point $Q \in V$, $Q \neq (-a_{31},-a_{32},a_{33})$, and we can proceed as in one of the cases covered so far. If $\delta_3 \geq a_{33}$, then there exists $\lambda \in \mathbb{Z}_{>0}$ such that $a_{33} >\delta_3 - \lambda a_{33} \geq 0$. Now, the cube $[[ (\lambda+1)a_{31}+\delta_1,(\lambda+1)a_{32}+\delta_2,\delta_3 - \lambda a_{33}]]$ has label $b$ and is in the region associated to some point $Q \in V$, $Q \neq (-a_{31},-a_{32},a_{33})$, hence we can proceed as in one of the cases covered so far.

    If $a_{31} = 0$, then $a_{33}d_3 = a_{32}d_2$, which implies $a_{32} \geq a_{22} = b_{22}$. We get 
    \begin{equation}
        a_{33}d_3 = b_{20}d_0 + b_{21}d_1 + (a_{32}-b_{22})d_2 + b_{23}d_3
    \end{equation}
    where $(b_{20},b_{21},0,b_{23}) \in \varphi^{-1}(b_{22}d_2)$ with $b_{20} > 0$, which contradicts $a_{33} < b_{33}$.

    The case $a_{32}=0$ follows analogously. \qedhere

\end{itemize}
\end{proof}

\begin{lemma}\label{lemma_catenary_help_primary}
    Suppose that $S$ is primary.  If $(-\lambda_1,\lambda_2,\lambda_3) \in V$ satisfies
    $$ \lambda_0d_0 + \lambda_1d_1 = \lambda_2d_2 + \lambda_3d_3$$
    with $\lambda_i \in \mathbb{Z}_{>0}$ and $\lambda_0 < a_{00}$, then for any $(\mu_0,\mu_1,\mu_2,\mu_3) \in \varphi^{-1}(\lambda_0d_0 + \lambda_1d_1)$ we have $(\mu_0,\mu_1,\mu_2,\mu_3) = (0,0,\lambda_2,\lambda_3)$ or $\mu_2=\mu_3 = 0$. An analogous statement holds for $(\lambda_1,-\lambda_2,\lambda_3)$ and $(\lambda_1,\lambda_2,-\lambda_3) \in V$.
\end{lemma}

\begin{proof}
Note that $\lambda_2< a_{22}$ and $\lambda_3 < a_{33}$ by the definition of $V$. 

Suppose $\mu_0>0$. If $\mu_2 > 0$, then by Lemma \ref{lemma_procedure_1}, no element of the Ap{\'e}ry set is a label in the region associated to $(-\mu_1,\lambda_2-\mu_2,\lambda_3-\mu_3)$. This contradicts the fact $[[0,\lambda_2-1, \lambda_3]] \in T$, which holds since $(-\lambda_1,\lambda_2,\lambda_3) \in V$. If $\mu_3 > 0$, we get an analogous contradiction, hence $\mu_2=\mu_3 = 0$.

Thus, suppose $\mu_0 = 0$. If $\mu_2 \geq \lambda_2$, then
\begin{equation}
     (\lambda_3-\mu_3)d_3=\mu_1d_1 + (\mu_2 - \lambda_2)d_2
\end{equation}
where all coefficients above are nonnegative. If $(\lambda_3-\mu_3) > 0$, then we get a contradiction with $\lambda_3 < a_{33}$. Hence $(\lambda_3-\mu_3) = 0$, which gives $(\mu_0,\mu_1,\mu_2,\mu_3) = (0,0,\lambda_2,\lambda_3)$. If $\mu_3 \geq \lambda_3$, we get the same thing; therefore, suppose that $\mu_2 < \lambda_2$ and $\mu_3 < \lambda_3$. We have
\begin{equation}
   \lambda_0d_0 + (\lambda_1-\mu_1)d_1 = \mu_2d_2 + \mu_3d_3.
\end{equation}
If $(\lambda_1-\mu_1) \leq 0$, then we get a contradiction with $\lambda_0 < a_{00}$, hence $(\lambda_1-\mu_1) > 0$. Since $\lambda_0>0$, by Lemma \ref{lemma_procedure_1}, no cube labeled with an element of the Ap{\'e}ry set is in the region associated to $(\mu_1 - \lambda_1,\mu_2,\mu_3)$. This contradicts the fact $[[0,\lambda_2-1,\lambda_3-1]] \in T$, which holds since $(-\lambda_1,\lambda_2,\lambda_3) \in V$.

The same argument applies to $(\lambda_1,-\lambda_2,\lambda_3)$ and $(\lambda_1,\lambda_2,-\lambda_3) \in V$.
\end{proof}

\begin{lemma}\label{lemma_catenary_primary}
    Suppose that $S$ is primary. If $(-\lambda_1,\lambda_2,\lambda_3) \in V$ satisfies
    $$ \lambda_0d_0 + \lambda_1d_1 = \lambda_2d_2 + \lambda_3d_3$$
    with $\lambda_i \in \mathbb{Z}_{>0}$, $\lambda_0 < a_{00}$, and $ \lambda_1 < a_{11}$, then $\varphi^{-1}(\lambda_0d_0 + \lambda_1d_1) = \{(\lambda_0,\lambda_1,0,0), (0,0,\lambda_2,\lambda_3) \} $. An analogous statement holds for $(\lambda_1,-\lambda_2,\lambda_3)$ and $(\lambda_1,\lambda_2,-\lambda_3) \in V$. 
\end{lemma}

\begin{proof}
Note that $\lambda_2< a_{22}$ and $\lambda_3 < a_{33}$ by the definition of $V$.
Suppose that there exists a factorization $ (\mu_0,\mu_1,\mu_2,\mu_3) \in \varphi^{-1}(\lambda_0d_0 + \lambda_1d_1)$ distinct from the two above. 
By Lemma \ref{lemma_catenary_help_primary}, we have $\mu_2 = \mu_3 = 0 $. 
If $\lambda_0 \geq \mu_0$, then
\begin{equation}
    (\lambda_0- \mu_0) d_0 =  (\mu_1 - \lambda_1)d_1
\end{equation}
where all coefficients above are nonnegative. Indeed, they are positive, since otherwise $(\mu_0,\mu_1,\mu_2,\mu_3) = (\lambda_0,\lambda_1,0,0)$. This contradicts $\lambda_0 < a_{00}$. If $\lambda_0 < \mu_0$, then $\lambda_1 > \mu_1$, and we get a contradiction with $\lambda_1 < a_{11}$. 

The same argument applies to $(\lambda_1,-\lambda_2,\lambda_3)$ and $(\lambda_1,\lambda_2,-\lambda_3) \in V$.
\end{proof}

\begin{lemma}\label{lemma_catenarycase1_primary}
    Suppose that $S$ is primary. Consider $(-\lambda_1,\lambda_2,\lambda_3) \in V$ satisfying
    $$ \lambda_0d_0 + \lambda_1d_1 = \lambda_2d_2 + \lambda_3d_3$$
    with $\lambda_i \in \mathbb{Z}_{>0}$, and $ \lambda_0 \geq a_{00}$ or $ \lambda_1 \geq a_{11}$. If $a_{00}d_0 \neq a_{11}d_1$, then $\lambda_0d_0 + \lambda_1d_1 \notin \textup{Betti}(S)$. Otherwise, if $a_{00}d_0 = a_{11}d_1$, then $\lambda_0d_0 + \lambda_1d_1 \in \textup{Betti}(S)$ and $\varphi^{-1}(\lambda_0d_0 + \lambda_1d_1) = \mathcal{F} \cup \{(0,0,\lambda_2,\lambda_3)\}$, where 
    $$\mathcal{F} =   \{(\lambda_0 + \lambda a_{00},\lambda_1 - \lambda a_{11},0,0) \mid \lambda \in \mathbb{Z}, \lambda_0 + \lambda a_{00} \geq 0, \lambda_1 - \lambda a_{11} \geq 0 \}. $$ 
    An analogous statement holds for $(\lambda_1,-\lambda_2,\lambda_3)$ and $(\lambda_1,\lambda_2,-\lambda_3) \in V$.
\end{lemma}

\begin{proof}
    Note that $\lambda_2< a_{22}$ and $\lambda_3 < a_{33}$ by the definition of $V$.
    
    Suppose $a_{00}d_0 \neq a_{11}d_1$ and $\lambda_0 \geq a_{00}$. Consider $(0,a_{01},a_{02},a_{03}) \in \varphi^{-1}(a_{00}d_0)$. If $\max \{a_{02},a_{03}\} > 0$, then
    \begin{equation}
        (\lambda_0,\lambda_1,0,0) \, \mathcal{R} \, (\lambda_0 - a_{00},\lambda_1 + a_{01},a_{02},a_{03}) \, \mathcal{R} \, (0,0,\lambda_2,\lambda_3)
    \end{equation}
    which results in only one $\mathcal{R}$-class. Hence $a_{02}=a_{03}=0$ and $a_{00}d_0 = a_{01}d_1$. Since $a_{00}d_0 \neq a_{11}d_1$, this gives $a_{01} > a_{11}$. Consider $(a_{10},0,a_{12},a_{13}) \in \varphi^{-1}(a_{11}d_1)$. If $a_{10}>0$, then we get a contradiction with the minimality of $a_{00}$, hence $a_{10}=0$. This implies  $\max \{a_{12},a_{13}\} > 0$ and hence 
    \begin{equation}
        (\lambda_0,\lambda_1,0,0) \, \mathcal{R} \, (\lambda_0 - a_{00},\lambda_1 + a_{01},0,0) \, \mathcal{R} \, (\lambda_0 - a_{00},\lambda_1 + a_{01}-a_{11},a_{12},a_{13}) \, \mathcal{R} \, (0,0,\lambda_2,\lambda_3)
    \end{equation}
    which results in only one $\mathcal{R}$-class. 
    
    If $a_{00}d_0 \neq a_{11}d_1$ and $\lambda_0 < a_{00}$, then $\lambda_1 \geq a_{11}$ and we can apply the same argument as above. Thus the first result follows.

    Suppose $a_{00}d_0 = a_{11}d_1$. Clearly $\mathcal{F} \subseteq \varphi^{-1}(\lambda_0d_0 + \lambda_1d_1)$. Suppose that there exists $(\mu_0,\mu_1,\mu_2,\mu_3) \in \varphi^{-1}(\lambda_0d_0 + \lambda_1d_1)$ that is not in $\mathcal{F} \cup \{(0,0,\lambda_2,\lambda_3)\}$.  By Lemma \ref{lemma_catenary_help_primary}, we get $\mu_2=\mu_3 = 0$ --- we can apply this lemma, because in $\mathcal{F}$, there exists a factorization that satisfies its assumptions. Then, as $(\mu_0,\mu_1,0,0) \notin \mathcal{F}$, we can find $\lambda \in \mathbb{Z}$ such that $\lambda_0 + (\lambda+1)a_{00} > \mu_0 > \lambda_0 + \lambda a_{00}$. This gives
    \begin{equation}
        \lambda_0d_0 + \lambda_1d_1 = (\mu_0 - \lambda a_{00})d_0 + (\mu_1 + \lambda a_{11})d_1 \implies (\mu_0 -\lambda a_{00} - \lambda_0)d_0 = (\lambda_1 - \mu_1 - \lambda a_{11})d_1  
    \end{equation}
    which gives a contradiction, since $ a_{00} > (\mu_0 - \lambda_0 - \lambda a_{00}) > 0$.

    The same argument applies to $(\lambda_1,-\lambda_2,\lambda_3)$ and $(\lambda_1,\lambda_2,-\lambda_3) \in V$.
\end{proof}

Lemma \ref{lemma_catenary_minrelations} gives $a_{ii}d_i \in \textup{Betti}(S)$ for $i = 0,1,2,3$. Combined with Lemmas \ref{lemma_catenary_primary} and \ref{lemma_catenarycase1_primary}, we know exactly which elements of $U$ are Betti elements. Moreover, we get that every Betti element that corresponds to a $(2,2)$-type point has exactly two $\mathcal{R}$-classes.

\begin{lemma}\label{lemma_catenaryT_primary}
    Suppose that $S$ is primary. If $(-\lambda_1,\lambda_2,\lambda_3) \in V$ satisfies
    $$ \lambda_0d_0 + \lambda_1d_1 = \lambda_2d_2 + \lambda_3d_3 \notin \textup{Betti}(S)$$
    with $\lambda_i \in \mathbb{Z}_{>0}$, then $\lambda_0 = a_{00}$ and $\lambda_1 < a_{11}$. An analogous statement holds for $(\lambda_1,-\lambda_2,\lambda_3)$ and $(\lambda_1,\lambda_2,-\lambda_3) \in V$. 
\end{lemma}

\begin{proof}

The condition $\lambda_0d_0 + \lambda_1d_1 \notin \textup{Betti}(S)$ implies $a_{00}d_0 \neq a_{11}d_1$, and $\lambda_0 \geq a_{00}$ or $\lambda_1 \geq a_{11}$, by Lemmas \ref{lemma_catenary_primary} and \ref{lemma_catenarycase1_primary}. Also, note that $\lambda_2< a_{22} $ and $\lambda_3 < a_{33} $ by the definition of $V$.

Suppose $\lambda_1 \geq a_{11}$. For $(a_{10},0,a_{12},a_{13}) \in \varphi^{-1}(a_{11}d_1)$, we have
\begin{equation}
     (\lambda_0+a_{10})d_0 + (\lambda_1-a_{11})d_1 =  (\lambda_2-a_{12})d_2 + (\lambda_3-a_{13})d_3
\end{equation}
where $(\lambda_1-a_{11}) \geq 0$. 
If $(\lambda_2-a_{12}) < 0$, then $\lambda_3 \geq (\lambda_3-a_{13}) \geq a_{33}$, which is a contradiction. If $(\lambda_3-a_{13}) < 0$, we get the same contradiction. Hence $(\lambda_2-a_{12}),(\lambda_3-a_{13}) \geq 0$.
If $\max \{a_{12},a_{13}\} > 0$, then by Lemma \ref{lemma_procedure_1}, no cube labeled with an element of the Ap{\'e}ry set is in the region associated to $(-\lambda_1+a_{11},\lambda_2-a_{12},\lambda_3-a_{13})$. This contradicts the fact $[[0,\lambda_2-1,\lambda_3-1]] \in T$, which holds since $(-\lambda_1,\lambda_2,\lambda_3) \in V$. 
Hence $a_{12} =a_{13}= 0$ and $a_{11}d_1 = a_{10}d_0$. This gives $a_{10} > a_{00}$, as $a_{00}d_0 \neq a_{11}d_1$. For $(0,a_{01},a_{02},a_{03}) \in \varphi^{-1}(a_{00}d_0)$, we have
\begin{equation}
    a_{11}d_1 = (a_{10}-a_{00})d_0 + a_{01}d_1 + a_{02}d_2 + a_{03}d_3
\end{equation}
which gives $a_{01} = 0$, as otherwise we would get a contradiction with the minimality of $a_{11}$. Therefore, we have $\max \{ a_{02},a_{03}\} > 0$ and
\begin{equation}
    (\lambda_0+a_{10}-a_{00})d_0 + (\lambda_1-a_{11} + a_{01})d_1 =  (\lambda_2-a_{02})d_2 + (\lambda_3-a_{03})d_3.
\end{equation}
Since $(\lambda_0+a_{10}-a_{00}) > 0$, no cube labeled with an element of the Ap{\'e}ry set is in the region associated to $(-\lambda_1+a_{11} - a_{01},\lambda_2-a_{02},\lambda_3-a_{03})$, by Lemma \ref{lemma_procedure_1}. Again, this contradicts the fact $[[0,\lambda_2-1,\lambda_3-1]] \in T$. Therefore $\lambda_1 <a_{11}$. 

If $\lambda_0 > a_{00}$, then for $(0,a_{01},a_{02},a_{03}) \in \varphi^{-1}(a_{00}d_0)$, we have
\begin{equation}
     (\lambda_0-a_{00})d_0 + (\lambda_1+a_{01})d_1 =  (\lambda_2-a_{02})d_2 + (\lambda_3-a_{03})d_3
\end{equation} 
where $(\lambda_0-a_{00}) > 0$. Proceeding as in the previous case, we obtain $a_{02}=a_{03}=0$ and $a_{01} > a_{11}$. Then, following the same approach further, we get $\max \{a_{12},a_{13}\} > 0$ and at last a contradiction. Therefore $\lambda_0 = a_{00}$.

The same argument applies to $(\lambda_1,-\lambda_2,\lambda_3)$ and $(\lambda_1,\lambda_2,-\lambda_3) \in V$.
\end{proof}

\subsection{Secondary case}

\begin{lemma}\label{lemmabettifind_secondary}
   Suppose that $S$ satisfies $a_{11} = b_{11}$ and $a_{22}d_2 = a_{33}d_3$. If $b \in \textup{Betti}(S)$ labels a cube in $R$, then $b = a_{22}d_2 = a_{33}d_3$.
\end{lemma}
    
\begin{proof}
    Consider a factorization $(0,\lambda_1,\lambda_2, \lambda_3) \in \varphi^{-1}(b)$ with $[[\lambda_1,\lambda_2,\lambda_3]] \in R$. Let $(\lambda_0',\lambda_1',\lambda_2', \lambda_3')$ be a distinct factorization of $b$. We have
    \begin{equation}\label{betticharacter_eq0a}
        (\lambda_1-\lambda_1')d_1 +  (\lambda_2-\lambda_2')d_2 +  (\lambda_3-\lambda_3')d_3 = \lambda_0'd_0.
    \end{equation}
    \begin{itemize}
    \item 
    If all coefficients on the left-hand side of \eqref{betticharacter_eq0a} are nonnegative, then $\lambda_0'>0$, as the two factorizations are distinct. This contradicts $[[\lambda_1,\lambda_2,\lambda_3]] \in R$ by Lemma \ref{lemma_procedure_1}.

    \item 
    Suppose that exactly one coefficient on the left-hand side of \eqref{betticharacter_eq0a} is nonnegative. If $(\lambda_1-\lambda_1')$ is the nonnegative coefficient, then $\lambda_1 \geq  (\lambda_1-\lambda_1') \geq a_{11}=b_{11}$, which contradicts $[[\lambda_1,\lambda_2,\lambda_3]] \in R$. Therefore, without loss of generality, suppose that $(\lambda_2-\lambda_2')$ is the nonnegative coefficient. This gives 
    \begin{equation}
          (\lambda_2-\lambda_2')d_2    = \lambda_0'd_0 + (\lambda_1'-\lambda_1)d_1 + (\lambda_3'-\lambda_3)d_3
    \end{equation}
    where every coefficient above is nonnegative, which implies $ (\lambda_2-\lambda_2') \geq a_{22}$ (since $(\lambda_2-\lambda_2') \neq 0$, as $(\lambda_3'-\lambda_3) > 0$). If $\lambda_0' > 0$, then $\lambda_2 \geq  (\lambda_2-\lambda_2') \geq b_{22}$, which contradicts $[[\lambda_1,\lambda_2,\lambda_3]] \in R$, hence $\lambda_0' = 0$.  
    There exists $\lambda \in \mathbb{N}$ such that $a_{22} > \lambda_2-\lambda_2' - \lambda a_{22} \geq 0$. Using $a_{22}d_2 = a_{33}d_3$, we obtain 
    \begin{equation}
          (\lambda_2-\lambda_2'- \lambda a_{22} )d_2    =  (\lambda_1'-\lambda_1)d_1 + (\lambda_3'-\lambda_3-\lambda a_{33} )d_3.
    \end{equation}
    If $(\lambda_3'-\lambda_3-\lambda a_{33} ) > 0$, then we get a contradiction with $a_{22} > \lambda_2-\lambda_2' - \lambda a_{22}$, hence every coefficient below is nonnegative.
    \begin{equation}
          (\lambda_2-\lambda_2'- \lambda a_{22} )d_2 + ( \lambda_3 -\lambda_3'+ \lambda a_{33})d_3    =  (\lambda_1'-\lambda_1)d_1.
    \end{equation}
    If $(\lambda_1'-\lambda_1) >0$, then $\lambda_1' \geq (\lambda_1'-\lambda_1) \geq a_{11} = b_{11}$. By Theorem \ref{theoremprocedure}, $R$ is the collection of all cubes labeled with an element of the Ap{\'e}ry set. This implies that $b$ is an element of the Ap{\'e}ry set, as it labels $[[\lambda_1,\lambda_2,\lambda_3]] \in R$. However, $b$ also labels $[[\lambda_1',\lambda_2',\lambda_3']]$ (recall $\lambda_0' = 0$), which is not in $R$ since $\lambda_1' \geq b_{11}$ --- contradiction.

    Hence $(\lambda_1'-\lambda_1) = 0$ and $(\lambda_0',\lambda_1',\lambda_2', \lambda_3') = (0,\lambda_1,\lambda_2 - \lambda a_{22} , \lambda_3 + \lambda a_{33} )$. As $(\lambda_0',\lambda_1',\lambda_2', \lambda_3') \in \varphi^{-1}(b)$ was chosen arbitrarily, we obtain
    \begin{equation}\label{betticharacter_factor}
        \varphi^{-1}(b) = \{ (0,\lambda_1,\lambda_2 - \lambda a_{22} , \lambda_3 + \lambda a_{33} ) \mid \lambda \in \mathbb{Z}, \lambda_2 - \lambda a_{22} \geq 0 , \lambda_3 + \lambda a_{33} \geq 0\}.
    \end{equation}
    Therefore $b = a_{22}d_2 = a_{33}d_3$, as otherwise $b$ would have only one $\mathcal{R}$-class, which would contradict $b \in \textup{Betti}(S)$.

    \item 
    Suppose that exactly two coefficients on the left-hand side of \eqref{betticharacter_eq0a} are nonnegative. 
    
    If $(\lambda_1-\lambda_1') < 0$, then
    \begin{equation}
        (\lambda_2-\lambda_2')d_2 +  (\lambda_3-\lambda_3')d_3 = \lambda_0'd_0 +  (\lambda_1'-\lambda_1)d_1
    \end{equation}
    where every coefficient above is nonnegative. If $\lambda_0' > 0$, we get a contradiction with $[[\lambda_1,\lambda_2,\lambda_3]] \in  R$ by Lemma \ref{lemma_procedure_1}. Hence $\lambda_0' = 0$, which implies $\lambda_1' \geq (\lambda_1'-\lambda_1) \geq a_{11} = b_{11}$. As before, $b$ labels $[[\lambda_1',\lambda_2',\lambda_3']]$, which is not in $R$ since $\lambda_1' \geq b_{11}$, contradicting Theorem \ref{theoremprocedure}.  

    Suppose $(\lambda_2-\lambda_2')<0$ (the case $(\lambda_3-\lambda_3') < 0$ follows analogously). This gives
    \begin{equation}
        (\lambda_1-\lambda_1')d_1 +  (\lambda_3-\lambda_3')d_3 = \lambda_0'd_0 +  (\lambda_2'-\lambda_2)d_2
    \end{equation}
    where every coefficient above is nonnegative. If $\lambda_0' > 0$, we get a contradiction with $[[\lambda_1,\lambda_2,\lambda_3]] \in R$ by Lemma \ref{lemma_procedure_1}. Hence $\lambda_0' = 0$, which implies $\lambda_2' \geq (\lambda_2'-\lambda_2) \geq a_{22}$.  There exists $\lambda \in \mathbb{N}$ such that $a_{22} > \lambda_2'-\lambda_2 - \lambda a_{22} \geq 0$. Using $a_{22}d_2 = a_{33}d_3$, we obtain 
    \begin{equation}
               (\lambda_1-\lambda_1')d_1 + (\lambda_3-\lambda_3'-\lambda a_{33} )d_3 = (\lambda_2'-\lambda_2- \lambda a_{22} )d_2.
    \end{equation}
    If $(\lambda_3-\lambda_3'-\lambda a_{33} ) > 0$, then we get a contradiction with $a_{22} > \lambda_2'-\lambda_2 - \lambda a_{22}$, hence every coefficient below is nonnegative.
    \begin{equation}
               (\lambda_1-\lambda_1')d_1 = (\lambda_3'-\lambda_3+\lambda a_{33} )d_3 + (\lambda_2'-\lambda_2- \lambda a_{22} )d_2.
    \end{equation}
     If $(\lambda_1-\lambda_1')$ is positive, then we get $\lambda_1 \geq (\lambda_1-\lambda_1') \geq a_{11} = b_{11}$, which contradicts $[[\lambda_1,\lambda_2,\lambda_3]] \in R$. 
     Hence $(\lambda_1-\lambda_1') = 0$, $(\lambda_0',\lambda_1',\lambda_2', \lambda_3') = (0,\lambda_1,\lambda_2 + \lambda a_{22} , \lambda_3 - \lambda a_{33} )$, and we obtain \eqref{betticharacter_factor}. This again gives $b =a_{22}d_2 =a_{33}d_3$. \qedhere   
    \end{itemize}  
\end{proof}

Let $V' \subset \mathbb{Z}^3$ be a minimal set (with respect to inclusion) of points $(x,y,z)$ such that $xd_1 + yd_2 +zd_3$ is a positive multiple of $d_0$ and such that removing from the initial collection of cubes the regions associated to the points of $V'$ yields $R$. 
Define
$$ U' = \{ \max\{ x,0\} d_1 + \max\{y,0\}d_2 + \max\{z,0\}d_3 \mid (x,y,z) \in V' \}.$$
Note that $U'$ is independent of the choice of $V'$. Similarly as for $V$, assume that every point $(x,y,z) \in V'$ such that
$$ \max\{ x,0\} d_1 + \max\{y,0\}d_2 + \max\{z,0\}d_3 \notin \{(a_{00}d_0,b_{11}d_1,b_{22}d_2,b_{33}d_3)\}$$
is a $(2,2)$-type point. Such $V'$ exists by Theorem \ref{theoremprocedure}.

\begin{lemma}\label{lemmaU_secondary}
    Suppose that $S$ satisfies $a_{11} = b_{11}$ and $a_{22}d_2 = a_{33}d_3$. Then $ U' \cup \{a_{33}d_3\} = U \cup \{b_{33}d_3\}$.
\end{lemma}

\begin{proof}
    Recall that the boundary values of $R$ and $T$ in the $x$-, $y$-, and $z$-directions are $b_{11}$, $b_{22}$, $b_{33}$ and $b_{11}$, $b_{22}$, $a_{33}$, respectively.

    Consider $(-\lambda_1,\lambda_2,\lambda_3) \in V'$ where $\lambda_i > 0$ and $\lambda_3 \geq a_{33}$. There exists $\lambda \in \mathbb{N}$ such that $a_{33} > \lambda_3 - \lambda a_{33} \geq 0$. Then, we have $(-\lambda_1,\lambda_2 + \lambda a_{22},\lambda_3-\lambda a_{33}) \in V$, since $T = R \cap \{z \leq a_{33}\}$ and $[[a,b,c]] \in R \cap \{y \geq a_{22}\}$ if and only if $[[a,b-a_{22},c+a_{33}]] \in R \cap \{z \geq a_{33}\}$ by Proposition \ref{prop_secondary}.
    
    We proceed analogously for any $(\lambda_1,-\lambda_2,\lambda_3) \in V'$ with $\lambda_3 \geq a_{33}$  (and $\lambda_i > 0$). We also have that any $(\lambda_1,\lambda_2,-\lambda_3) \in V'$ ($\lambda_i > 0$) is in $V$, since $T = R \cap \{z \leq a_{33}\}$. Consequently, $ U' \cup \{a_{33}d_3\} = U \cup \{b_{33}d_3\}$.    
\end{proof}

\begin{thm}\label{theorembettifind_secondary}
If $S$ is secondary, then $\textup{Betti}(S) \subseteq U \cup \{b_{33}d_3\}$. 
\end{thm}

\begin{proof}
    Suppose that there exists an element $b \in \text{Betti}(S)$ that is not in $U\cup \{b_{33}d_3\}$, which by Lemma \ref{lemmaU_secondary} equals $U' \cup \{a_{33}d_3\}$.
    There exists a factorization $(0,\kappa_1,\kappa_2,\kappa_3) \in \varphi^{-1}(b)$; otherwise, if every factorization $z \in \varphi^{-1}(b)$ had $0 \in \text{supp}(z)$, then all factorizations would be $\mathcal{R}$-related, which cannot happen for a Betti element. By Lemma \ref{lemmabettifind_secondary}, we have $[[\kappa_1,\kappa_2,\kappa_3]] \notin R$; therefore $[[\kappa_1,\kappa_2,\kappa_3]]$ is in the region associated to some point $P \in V'$.  
\begin{itemize}
    \item 
    If $P =(-\lambda_1,\lambda_2,\lambda_3)$, where
    \begin{equation}
        \lambda_0d_0 + \lambda_1d_1 = \lambda_2d_2 + \lambda_3d_3
    \end{equation}
    with $\lambda_i \in \mathbb{Z}_{>0}$, we get
     \begin{equation}
        b =\lambda_0d_0 + (\lambda_1+\delta_1)d_1 + \delta_2d_2 + \delta_3d_3  =  \delta_1d_1 + (\lambda_2+\delta_2)d_2 + (\lambda_3+\delta_3)d_3
    \end{equation}
    for some $\delta_i \in \mathbb{N}$ with $\max \{\delta_1,\delta_2,\delta_3\} > 0$. The factorizations $z_1 =(\lambda_0,\lambda_1+\delta_1,\delta_2,\delta_3)$ and $z_2=(0, \delta_1,\lambda_2+\delta_2,\lambda_3+\delta_3)$ have a nonempty intersection of supports, and hence they are $\mathcal{R}$-related. Moreover, we have $0,1,2,3 \in \textup{supp}(z_1) \cup \textup{supp}(z_2)$, so every factorization of $b$ will be $\mathcal{R}$-related to $z_1$ and $z_2$. This contradicts $b \in \textup{Betti}(S)$, as it must have at least two distinct $\mathcal{R}$-classes. 
    
    The cases $P=(\lambda_1,-\lambda_2,\lambda_3)$ and $P=(\lambda_1,\lambda_2,-\lambda_3)$ follow analogously.

\item
    If $P =(-b_{21},b_{22},-b_{23})$, where $(b_{20},b_{21},0,b_{23}) \in \varphi^{-1} (b_{22}d_2)$ with $b_{20}>0$, we get
    \begin{equation}
        b = \delta_1d_1 +   (b_{22}+\delta_2)d_2 + \delta_3d_3  =  b_{20}d_0 + (b_{21}+\delta_1)d_1+ \delta_2d_2 + (b_{23}+\delta_3)d_3
    \end{equation}
    for some $\delta_i \in \mathbb{N}$ with $\max \{\delta_1,\delta_2,\delta_3\} > 0$. The factorizations $z_1 =(0,\delta_1,b_{22}+\delta_2,\delta_3)$ and $z_2 =(b_{20},b_{21}+\delta_1,\delta_2,b_{23}+\delta_3)$ are $\mathcal{R}$-related. 
    Because $b \in \textup{Betti}(S)$, there exists a factorization $z \in \varphi^{-1}(b)$ that is not $\mathcal{R}$-related to them. Let $z_3 = (0,\delta_1,b_{22}+\delta_2 - a_{22},\delta_3+a_{33}) \in \varphi^{-1}(b)$ (as $a_{22}d_2 = a_{33}d_3$). We have $z_1 \, \mathcal{R} \, z_2 \, \mathcal{R} \,   z_3 $ and $0,2,3 \in \textup{supp}(z_1) \cup \textup{supp}(z_2) \cup \textup{supp}(z_3)$. This implies $z = (0,\mu_1,0,0)$. If $\mu_1 = a_{11} = b_{11}$, then $b = b_{11}d_1 \in U$, hence $\mu_1 > a_{11}$. This gives $z \, \mathcal{R} \, (a_{00},\mu_1-a_{11},0,0) \, \mathcal{R} \, z_2$, which is a contradiction.

    The case $P = (-b_{31},-b_{32},b_{33})$ follows analogously.
\item
    If $P = (a_{01},a_{02},a_{03})$, where $(0,a_{01},a_{02},a_{03}) \in \varphi^{-1}(a_{00}d_0)$, then as $S$ is secondary, we have $(a_{01},a_{02},a_{03}) = (a_{11},0,0)$. Thus this case reduces to the next one.
\item    
    If $P =(b_{11},0,0)=(a_{11},0,0)$, we get
     \begin{equation}
        b =a_{00}d_0 + \delta_1d_1 + \delta_2d_2 + \delta_3d_3  =  (a_{11}+\delta_1)d_1 + \delta_2d_2 + \delta_3d_3
    \end{equation}
    for some $\delta_i \in \mathbb{N}$ with $\max \{\delta_1,\delta_2,\delta_3\} > 0$. The factorizations $z_1 =(a_{00},\delta_1,\delta_2,\delta_3)$ and $z_2 =(0,a_{11}+\delta_1,\delta_2,\delta_3)$ have a nonempty intersection of supports, and hence they are $\mathcal{R}$-related. Because $b \in \textup{Betti}(S)$, there exists a factorization $z \in \varphi^{-1}(b)$ that is not $\mathcal{R}$-related to them. As $0,1 \in \textup{supp}(z_1) \cup \textup{supp}(z_2)$, we have $z = (0,0,\mu_2,\mu_3)$. If $[[0,\mu_2,\mu_3]] \in R$, then $b =a_{22}d_2=a_{33}d_3 \in U$ by Lemma \ref{lemmabettifind_secondary}. Therefore $[[0,\mu_2,\mu_3]] \notin R$, so this cube is in the region associated to some $Q \in V'$, where $Q \neq (b_{11},0,0)$. Since $[[0,\mu_2,\mu_3]]$ is labeled with $b$, we can apply the same argument as in the earlier cases. \qedhere
\end{itemize}
\end{proof}

\begin{lemma}\label{lemma_catenary_lambda1_secondary}
    Suppose that $S$ is secondary.  If $(-\lambda_1,\lambda_2,\lambda_3) \in V$ satisfies
    $$ \lambda_0d_0 + \lambda_1d_1 = \lambda_2d_2 + \lambda_3d_3$$
    with $\lambda_i \in \mathbb{Z}_{>0}$, then $\lambda_0d_0 + \lambda_1d_1 \in \textup{Betti}(S)$ and $\varphi^{-1}(\lambda_0d_0 + \lambda_1d_1) = \mathcal{F}_1 \cup \mathcal{F}_2$, where 
    \begin{align*}
        &\mathcal{F}_1 =  \{(\lambda_0 + \lambda a_{00},\lambda_1 - \lambda a_{11},0,0) \mid \lambda \in \mathbb{Z}, \lambda_0 + \lambda a_{00} \geq 0,\lambda_1 - \lambda a_{11} \geq 0\}, \\
        &\mathcal{F}_2 =  \{(0,0,\lambda_2 + \lambda a_{22},\lambda_3 - \lambda a_{33}) \mid \lambda \in \mathbb{Z}, \lambda_2 + \lambda a_{22} \geq 0,\lambda_3 - \lambda a_{33} \geq 0\}.
    \end{align*}
  
\end{lemma}

\begin{proof}
    Clearly $\mathcal{F}_1 \cup \mathcal{F}_2 \subseteq \varphi^{-1}(\lambda_0d_0 + \lambda_1d_1)$.
    Consider any $(\mu_0,\mu_1,\mu_2,\mu_3) \in \varphi^{-1}(\lambda_0d_0 + \lambda_1d_1)$.

    \begin{itemize}
    \item 

    Suppose $\mu_0 > 0$. If $\mu_2 > 0$, then by Lemma \ref{lemma_procedure_1}, no cube labeled with an element of the Ap{\'e}ry set is in the region associated to $(-\mu_1,\lambda_2-\mu_2,\lambda_3-\mu_3)$. This contradicts the fact $[[0,\lambda_2-1, \lambda_3]] \in T$, which holds since $(-\lambda_1,\lambda_2,\lambda_3) \in V$. If $\mu_3 > 0$, we get an analogous contradiction, hence $\mu_2=\mu_3 = 0$. If $(\mu_0,\mu_1,0,0) \notin \mathcal{F}_1$, then there exists $\lambda \in \mathbb{Z}$ such that $\lambda_0 + (\lambda+1) a_{00} > \mu_0  > \lambda_0 + \lambda a_{00}$. Now, we have
    \begin{equation}
        \lambda_0d_0 + \lambda_1d_1 = (\mu_0 - \lambda a_{00})d_0 + (\mu_1 + \lambda a_{11})d_1 \implies (\mu_0  - \lambda_0- \lambda a_{00})d_0 = (\lambda_1 - \mu_1 - \lambda a_{11})d_1
    \end{equation}
    which gives a contradiction, since $a_{00} > (\mu_0 - \lambda_0 - \lambda a_{00}) > 0$.

    \item
    Suppose $\mu_0 = 0$ and $\mu_1 > 0$. If $\mu_2 \geq \lambda_2$, then we get
    \begin{equation}
         (\lambda_3-\mu_3)d_3=\mu_1d_1 + (\mu_2 - \lambda_2)d_2
    \end{equation}
    where all coefficients above are nonnegative. Moreover, we have $(\lambda_3-\mu_3) > 0$, as $\mu_1 > 0$. Hence $(\lambda_3-\mu_3) \geq a_{33}$, so there exists $\lambda \in \mathbb{Z}_{>0}$ such that $a_{33} >(\lambda_3-\mu_3 - \lambda a_{33}) \geq 0$. We have 
    \begin{equation}
         (\lambda_3-\mu_3-\lambda a_{33})d_3=\mu_1d_1 + (\mu_2 - \lambda_2-\lambda a_{22})d_2.
    \end{equation}
    If $(\mu_2 - \lambda_2-\lambda a_{22}) \geq 0$, then as $\mu_1 > 0$ we have $(\lambda_3-\mu_3-\lambda a_{33}) > 0$, which gives $(\lambda_3-\mu_3-\lambda a_{33}) \geq a_{33}$ --- contradiction. Hence $(\mu_2 - \lambda_2-\lambda a_{22}) < 0$, which gives $\mu_1 \geq a_{11}$. Then, repeating the approach of the first bullet-point to the factorization $(a_{00},\mu_1 - a_{11},\mu_2,\mu_3) \in \varphi^{-1}(\lambda_0d_0 + \lambda_1d_1)$, yields $\mu_2 = \mu_3 = 0$ and $(a_{00},\mu_1 - a_{11},0,0) \in \mathcal{F}_1$, so $(\mu_0,\mu_1,\mu_2,\mu_3) =(0,\mu_1,0,0) \in \mathcal{F}_1$ also.
    
    If $\mu_3 \geq \lambda_3$, we proceed analogously; therefore, suppose that $\mu_2 < \lambda_2$ and $\mu_3 < \lambda_3$. We have 
    \begin{equation}\label{lemma_catenary_lambda1latter_secondary_eq1}
        \lambda_0d_0 = (\mu_1-\lambda_1)d_1 + \mu_2d_2 + \mu_3d_3. 
    \end{equation} 
     By Lemma \ref{lemma_procedure_1}, no cube labeled with an element of the Ap{\'e}ry set is in the region associated to $(\mu_1 - \lambda_1,\mu_2,\mu_3)$. If $(\mu_1 - \lambda_1) \leq 0$, this gives a contradiction with the fact $[[0,\lambda_2-1,\lambda_3-1]] \in T$, which holds since $(-\lambda_1,\lambda_2,\lambda_3) \in V$. Hence $(\mu_1 - \lambda_1) > 0$, which gives $\lambda_0 \geq a_{00}$. If $\lambda_0 = a_{00}$, then $\mu_2 = \mu_3 = 0$ and $\mu_1 = \lambda_1 + a_{11}$, since $\varphi^{-1}(a_{00}d_0) = \{(a_{00},0,0,0),(0,a_{11},0,0)\}$. This gives $(\mu_0,\mu_1,\mu_2,\mu_3) \in \mathcal{F}_1$, so suppose $\lambda_0 > a_{00}$. We have
    \begin{equation}
        (\lambda_0-a_{00})d_0 = (\mu_1-\lambda_1-a_{11})d_1 + \mu_2d_2 + \mu_3d_3. 
    \end{equation}
    Now, we can repeat the same argument used for \eqref{lemma_catenary_lambda1latter_secondary_eq1} to eventually obtain a contradiction or $(\mu_0,\mu_1,\mu_2,\mu_3) \in \mathcal{F}_1$. 
    
    \item 
    Suppose $\mu_0 =\mu_1 = 0$. If $(0,0,\mu_2,\mu_3) \notin \mathcal{F}_2$, then there exists $\lambda \in \mathbb{Z}$ such that $ \lambda_2  + (\lambda+1) a_{22} > \mu_2 > \lambda_2  + \lambda a_{22}$. Now, we have
    \begin{equation}
        \lambda_2d_2 + \lambda_3d_3 = (\mu_2 - \lambda a_{22})d_2 + (\mu_3 + \lambda a_{33})d_3 \implies (\mu_2 - \lambda_2 - \lambda a_{22} )d_2 = (\lambda_3 - \mu_3 - \lambda a_{33})d_3
    \end{equation}
    which gives a contradiction, since $ a_{22} > (\mu_2 - \lambda_2  - \lambda a_{22}) > 0$. \qedhere
     \end{itemize}
\end{proof}

\begin{lemma}\label{lemma_catenary_lambda23_secondary}

    Suppose that $S$ is secondary.  Consider $(\lambda_1,-\lambda_2,\lambda_3) \in V$ satisfying
    $$ \lambda_0d_0 + \lambda_2d_2 = \lambda_1d_1 + \lambda_3d_3$$
    with $\lambda_i \in \mathbb{Z}_{>0}$. If $\lambda_0 \geq a_{00}$ or $\lambda_2 \geq a_{22}$, then $\lambda_0d_0 + \lambda_2d_2 \notin \textup{Betti}(S)$.
    Otherwise, if $\lambda_0 < a_{00}$ and $\lambda_2 < a_{22}$, then $\lambda_0d_0 + \lambda_2d_2 \in \textup{Betti}(S)$ and $\varphi^{-1}(\lambda_0d_0 + \lambda_2d_2) = \{(\lambda_0,0,\lambda_2,0),(0,\lambda_1,0,\lambda_3)\}$. An analogous statement holds for $(\lambda_1,\lambda_2,-\lambda_3) \in V$.
\end{lemma}

\begin{proof}
    If $\lambda_0 \geq a_{00}$, then $(\lambda_0,0,\lambda_2,0) \, \mathcal{R} \, (\lambda_0-a_{00},a_{11},\lambda_2,0) \, \mathcal{R} \, (0,\lambda_1,0,\lambda_3)$, which implies that $\lambda_0d_0 + \lambda_2d_2$ has only one $\mathcal{R}$-class. If $\lambda_2 \geq a_{22}$ we can proceed analogously. Thus the first part of the lemma follows.

    Suppose that $\lambda_0 < a_{00}$ and $\lambda_2 < a_{22}$. Consider any $(\mu_0,\mu_1,\mu_2,\mu_3) \in \varphi^{-1}(\lambda_0d_0 + \lambda_2d_2)$.

    \begin{itemize}
        \item Suppose $\mu_0>0$. If $\mu_1 > 0$, then by Lemma \ref{lemma_procedure_1}, no cube labeled with an element of the Ap{\'e}ry set is in the region associated to $(\lambda_1-\mu_1,-\mu_2,\lambda_3-\mu_3)$. This contradicts the fact $[[\lambda_1-1,0, \lambda_3]] \in T$, which holds since $(\lambda_1,-\lambda_2,\lambda_3) \in V$. If $\mu_3 > 0$, we get an analogous contradiction, hence $\mu_1=\mu_3 = 0$. This gives
        \begin{equation}
            (\lambda_0- \mu_0) d_0 =  (\mu_2 - \lambda_2)d_2.
        \end{equation}
        If $(\lambda_0- \mu_0) = 0$, then $(\mu_0,\mu_1,\mu_2,\mu_3) = (\lambda_0,0,\lambda_2,0)$. If $(\lambda_0- \mu_0) > 0$, then $\lambda_0 \geq (\lambda_0- \mu_0) \geq a_{00}$ --- contradiction. If $(\lambda_0- \mu_0) < 0$, then $\lambda_2 \geq (\lambda_2- \mu_2) \geq a_{22}$ --- contradiction.

        \item
    Suppose $\mu_0 = 0$ and $\mu_2 > 0$. If $\mu_1 \geq \lambda_1$, then
    \begin{equation}
         (\lambda_3-\mu_3)d_3 = (\mu_1-\lambda_1)d_1 + \mu_2d_2
    \end{equation}
    where all coefficients above are nonnegative. If $(\lambda_3-\mu_3) = 0$, then $(\mu_0,\mu_1,\mu_2,\mu_3) = (0,\lambda_1,0,\lambda_3)$. Thus, suppose $(\lambda_3-\mu_3) > 0$, which implies $\lambda_3 \geq (\lambda_3-\mu_3) \geq a_{33}$. Now, as $\lambda_2 < a_{22}$, we have
    \begin{equation}\label{lemma_catenary_lambda23_secondary_eq2}
        \lambda_0d_0 = \lambda_1d_1 + (a_{22}-\lambda_2)d_2 + (\lambda_3 - a_{33})d_3
    \end{equation}
    where all coefficients above are nonnegative, which contradicts $\lambda_0 < a_{00}$. 
    
    If $\mu_3 \geq \lambda_3$, then
    \begin{equation}
         (\lambda_1-\mu_1)d_1 = (\mu_3-\lambda_3)d_3 + \mu_2d_2
    \end{equation}
    where all coefficients above are nonnegative. If $(\lambda_1-\mu_1) = 0$, then $(\mu_0,\mu_1,\mu_2,\mu_3) = (0,\lambda_1,0,\lambda_3)$. Thus, suppose $(\lambda_1-\mu_1) > 0$, which implies $\lambda_1 \geq (\lambda_1-\mu_1) \geq a_{11}$. Now, as $\lambda_0 < a_{00}$, we have
    \begin{equation}\label{lemma_catenary_lambda23_secondary_eq1}
        \lambda_2d_2 = (a_{00}-\lambda_0)d_0 + (\lambda_1 - a_{11})d_1 + \lambda_3d_3
    \end{equation}
    where all coefficients above are nonnegative, which contradicts $\lambda_2 < a_{22}$.

    Hence $\mu_1 < \lambda_1$ and $\mu_3 < \lambda_3$. We have
    \begin{equation}
         \lambda_0d_0 = \mu_1d_1 + (\mu_2 - \lambda_2)d_2 + \mu_3d_3.
    \end{equation}
     By Lemma \ref{lemma_procedure_1}, no cube labeled with an element of the Ap{\'e}ry set is in the region associated to $(\mu_1,\mu_2-\lambda_2,\mu_3)$. If $(\mu_2 - \lambda_2) \leq 0$, this gives a contradiction with the fact $[[\lambda_1-1,0,\lambda_3-1]] \in T$, which holds since $(\lambda_1,-\lambda_2,\lambda_3) \in V$. Hence $(\mu_2 - \lambda_2) > 0$, which gives $\lambda_0 \geq a_{00}$ --- contradiction.

    \item 
    Finally, suppose $\mu_0 = \mu_2 = 0$. We have
    \begin{equation}
        (\lambda_1 - \mu_1)d_1 = (\mu_3 - \lambda_3)d_3.
    \end{equation}
    If $(\lambda_1 - \mu_1) = 0$, then $(0,\mu_1,0,\mu_3) = (0,\lambda_1,0,\lambda_3)$. 
    If $(\lambda_1 - \mu_1) > 0$, then $\lambda_1 \geq (\lambda_1 - \mu_1) \geq a_{11}$ and we obtain \eqref{lemma_catenary_lambda23_secondary_eq1} with the same contradiction.
    Hence $(\lambda_1 - \mu_1) < 0$, which implies $\lambda_3 \geq (\lambda_3 - \mu_3) \geq a_{33} $. In this case we obtain \eqref{lemma_catenary_lambda23_secondary_eq2} leading to the same contradiction.
    \end{itemize}

The same argument applies to $(\lambda_1,\lambda_2,-\lambda_3) \in V$. 
\end{proof}

From Lemmas \ref{lemma_catenary_lambda1_secondary} and \ref{lemma_catenary_lambda23_secondary}, we know which elements of $U$ corresponding to $(2,2)$-type points are Betti elements. Moreover, we get that every Betti element that corresponds to a $(2,2)$-type point has exactly two $\mathcal{R}$-classes. We also know that $a_{ii}d_i \in \textup{Betti}(S)$ for $i = 0,1,2,3$ by Lemma \ref{lemma_catenary_minrelations}. However, we do not know whether $b_{22}d_2$ and $b_{33}d_3$ are Betti elements. This gap will not affect the results of Section \ref{section_type}, so we leave it unresolved.

\begin{lemma}\label{lemma_catenaryT_secondary_case1}
    Suppose that $S$ is secondary. No element $(\lambda_1,-\lambda_2,\lambda_3) \in V$ satisfies
    $$ \lambda_0d_0 + \lambda_2d_2 = \lambda_1d_1 + \lambda_3d_3 \notin \textup{Betti}(S)$$
    with $\lambda_i \in \mathbb{Z}_{>0}$.
\end{lemma}
\begin{proof}
By Lemma \ref{lemma_catenary_lambda23_secondary}, we have $\lambda_0 \geq a_{00}$ or $\lambda_2 \geq a_{22}$. Since $(\lambda_1,-\lambda_2,\lambda_3) \in V$ and $T = R \cap \{z \leq a_{33}\}$, we have $\lambda_1 < b_{11} = a_{11}$ and $\lambda_3 < a_{33}$. 

If $\lambda_0\geq a_{00}$, then 
\begin{equation}
     (\lambda_0-a_{00})d_0 + \lambda_2d_2 + (a_{11}-\lambda_1)d_1 =  \lambda_3d_3 
\end{equation}
where all above coefficients are nonnegative. Since $(a_{11}-\lambda_1) > 0$, we get $\lambda_3 \geq a_{33}$, which is a contradiction.

If $\lambda_2\geq a_{22}$, then
\begin{equation}
     \lambda_0d_0 + (\lambda_2-a_{22})d_2 + (a_{33}-\lambda_3)d_3 =  \lambda_1d_1 
\end{equation}
where all above coefficients are nonnegative. Since $(a_{33}-\lambda_3) > 0$, we get $\lambda_1 \geq a_{11}$, which is a contradiction.
\end{proof}

\begin{lemma}\label{lemma_catenaryT_secondary_case2}
    Suppose that $S$ is secondary. If $(\lambda_1,\lambda_2,-\lambda_3) \in V$ satisfies
    $$ \lambda_0d_0 + \lambda_3d_3 = \lambda_1d_1 + \lambda_2d_2 \notin \textup{Betti}(S)$$
    with $\lambda_i \in \mathbb{Z}_{>0}$, then $\lambda_0 = a_{00}$ and $\lambda_3 < a_{33}$.
\end{lemma}

\begin{proof}
By Lemma \ref{lemma_catenary_lambda23_secondary}, we have $\lambda_0 \geq a_{00}$ or $\lambda_3 \geq a_{33}$.  Since $(\lambda_1,-\lambda_2,\lambda_3) \in V$ and $T = R \cap \{z \leq a_{33}\}$, we have $\lambda_1 < b_{11} = a_{11}$.

If $\lambda_3\geq a_{33}$, then
\begin{equation}
     \lambda_0d_0   + (\lambda_3-a_{33})d_3=  \lambda_1d_1 + (\lambda_2-a_{22})d_2 
\end{equation}
where $(\lambda_3-a_{33}) \geq 0$.
If $(\lambda_2-a_{22}) < 0$, then $\lambda_1 \geq a_{11}$, which is a contradiction. Hence $(\lambda_2-a_{22}) \geq 0$. Since $\lambda_0 > 0$, no cube labeled with an element of the Ap{\'e}ry set is in the region associated to $(\lambda_1,\lambda_2-a_{22},-\lambda_3+a_{33})$ by Lemma \ref{lemma_procedure_1}. This contradicts the fact $[[\lambda_1,\lambda_2-1,0]] \in T$, which holds since $(\lambda_1,\lambda_2,-\lambda_3) \in V$. Hence $\lambda_3 < a_{33}$.

If $\lambda_0 > a_{00}$, then
\begin{equation}
     (\lambda_0-a_{00})d_0 + \lambda_3d_3 + (a_{11}-\lambda_1)d_1 =  \lambda_2d_2
\end{equation}
where all above coefficients are positive. This gives $\lambda_2 \geq b_{22}$, which contradicts $(\lambda_1,\lambda_2,-\lambda_3) \in V$. Hence $\lambda_0 = a_{00}$.
\end{proof}

\section{$t(S)$ and $\eta(S)$ }\label{section_type}

This section is devoted to proving Theorem \ref{theorem_type_main}. We begin with a series of lemmas relating the type to the set $V$, culminating in Theorem \ref{theorem_type1}. We then turn to results linking $V$ to $\eta(S)$, concluding with Corollary \ref{corollary_type_main}, from which Theorem \ref{theorem_type_main} follows.

Let $S = \langle d_0,d_1,d_2,d_3 \rangle$ be a numerical semigroup with embedding dimension four that is ordered. Let $T$, $R$, $V$, and $U$ be as in the previous sections.

\subsection{$t(S)$ and $V$}

A finite and nonempty subset $F$ of the initial collection of cubes is called a \emph{cube figure} if for every $(a,b,c) \in \mathbb{N}^3$ with $[[a,b,c]] \notin F$, no cube in the region associated to $(a,b,c)$ is in $F$. For a cube figure $F$, a cube $[[a,b,c]] \in F$ is called a \emph{corner} of $F$ if $[[a+1,b,c]],[[a,b+1,c]],[[a,b,c+1]] \notin F$. 

Recall from \eqref{eq_pseudo_frob} that
\begin{equation}
     PF(S) = \{s - d_0 \mid s \in \textup{Ap}(S,d_0), s + d_i \notin \textup{Ap}(S,d_0) \text{ for } i = 1,2,3\}.
\end{equation}
Thus, since $R$ is the collection of all cubes labeled with an element of the Ap{\'e}ry set, $PF(S)$ is the set of all labels of corners of $R$, and $t(S)$ is the number of such distinct labels.

\begin{lemma}\label{lemma_point0uniqueness}
    There exists at most one $(0,a_{01},a_{02},a_{03}) \in \varphi^{-1}(a_{00}d_0)$ that satisfies $(a_{01},a_{02},a_{03}) \in V \setminus \{(b_{11},0,0),(0,b_{22},0),(0,0,b_{33})\}$.
\end{lemma}
\begin{proof}
If $S$ is secondary, then $\varphi^{-1}(a_{00}d_0) = \{(a_{00},0,0,0), (0,a_{11},0,0)\}$ (and by definition $a_{11}=b_{11}$), so no such $(0,a_{01},a_{02},a_{03}) \in \varphi^{-1}(a_{00}d_0)$ exists in this case. 

Suppose that $S$ is primary. Suppose to the contrary, that $(0,a_{01},a_{02},a_{03})$ and $(0,a_{01}',a_{02}',a_{03}') \in \varphi^{-1}(a_{00}d_0)$ are two such distinct factorizations. Since $T \subseteq \{x \leq a_{11}, y \leq a_{22}, z \leq a_{33}\}$, it follows that $a_{0i}, a_{0i}' \leq a_{ii}$. If $(a_{01},a_{02},a_{03}) = (0,0,a_{33})$, then $a_{33}=b_{33}$ (as $a_{00}>0$), which contradicts $(a_{01},a_{02},a_{03}) \neq (0,0,b_{33})$. If  $(a_{01},a_{02},a_{03}) \in \{(a_{11},0,0) ,(0,a_{22},0) \}$, we get an analogous contradiction, therefore $a_{0i}, a_{0i}' < a_{ii}$.  We have
\begin{equation}
    (a_{01}-a_{01}')d_1 +  (a_{02}-a_{02}')d_2 + (a_{03}-a_{03}')d_3 = 0.
\end{equation}
Since $(0,a_{01},a_{02},a_{03}) \neq (0,a_{01}',a_{02}',a_{03}')$, not all coefficients above are zero.
Without loss of generality, suppose $(a_{01}-a_{01}') >0$, $(a_{02}-a_{02}') \leq 0$, and $(a_{03}-a_{03}') \leq 0$. This implies $a_{01} \geq (a_{01}-a_{01}') \geq a_{11}$, which is a contradiction.
\end{proof}

Let $V^*$ denote the set of all $(2,2)$-type points of $V$. By the definition of $V$ and Lemma \ref{lemma_point0uniqueness}, we get
\begin{equation}\label{type_assumption_on_V}
    V = V^* \sqcup \{(b_{11},-b_{12},-b_{13}), (-b_{21},b_{22},-b_{23}), (-a_{31},-a_{32},a_{33}), (a_{01},a_{02},a_{03})\}
\end{equation}
where $a_{ij},b_{ij} \in \mathbb{N}$ and $(0,a_{01},a_{02},a_{03}) \in \varphi^{-1}(a_{00}d_0)$. Such a point $(a_{01},a_{02},a_{03}) \in V$ exists, since $a_{00}d_0 \in \textup{Betti}(S) \subseteq U \cup \{b_{33}d_3\}$ by Theorems \ref{theorembettifind_primary} and \ref{theorembettifind_secondary} --- perhaps $(a_{01},a_{02},a_{03}) \in \{(b_{11},-b_{12},-b_{13})$, $(-b_{21},b_{22},-b_{23})$, $(-a_{31},-a_{32},a_{33}) \}$. 

The identity \eqref{type_assumption_on_V} yields the following lemma.

\begin{lemma}\label{lemma_point0connection}
    If $(\lambda_1,\lambda_2,\lambda_3) \in V$ with $\lambda_i \geq 0$ for $i=1,2,3$ and $\lambda_i > 0$ for at least two distinct indices $i \in \{1,2,3\}$, then $(0,\lambda_1,\lambda_2,\lambda_3) \in \varphi^{-1}(a_{00}d_0)$.
\end{lemma}

\begin{lemma}\label{lemma_cornerlabelRT}
    Every integer that labels a corner of $R$ also labels a corner of $T$.
\end{lemma}

\begin{proof}
Take $(a_{30},a_{31},a_{32},0) \in \varphi^{-1}(a_{33}d_3)$. If $a_{30}>0$, then $a_{33}=b_{33}$ and $R = T$, which yields the claim. Thus, suppose $a_{30} =0$.

Consider a corner $[[a,b,c]]$ of $R$ that is in the region $R \cap \{z \geq a_{33}\}$ (that is $c \geq a_{33}$). There exists $\lambda \in \mathbb{N}$ such that $a_{33} > c- \lambda a_{33} \geq 0$, for which we have $C = [[a + \lambda a_{31},b + \lambda a_{32},c- \lambda a_{33}]] \in  R \cap \{z \leq a_{33}\} = T$, since $C$ has the same label as $[[a,b,c]]$. 
We claim that $C$ is a corner of $T$. If not, then there exists $[[a',b',c']] \in T $ with $C <  [[a',b',c']]$. However, this implies $[[a,b,c]] < [[a' -\lambda a_{31},b'-\lambda a_{32},c' + \lambda a_{33} ]] \in R$, which contradicts the fact that $[[a,b,c]]$ is a corner of $R$.
\end{proof}

Let $t'$ denote the number of corners of $T$ and let $\mathcal{T}$ be the set of the corners of $T$ that are not corners of $R$. Since $T$ is an $L$-shape, the corners of $T$ have distinct labels. This fact, combined with Lemma \ref{lemma_cornerlabelRT}, yields $t' = | \mathcal{T}| + t(S)$.

For the semigroups $\langle 196,225,256,289 \rangle$ and $\langle 145,203,74,111 \rangle$ from Example \ref{example_T}, shown in Figures \ref{fig_cornersT_primary} and \ref{fig_cornersT_secondary}, respectively, the cubes shown in blue are all the corners of $T$ that are also corners of $R$, and the cubes shown in magenta are all the cubes in $\mathcal{T}$.

\begin{figure}[h]
    \centering
    \begin{minipage}[t]{0.48\textwidth}
        \centering
        \includegraphics[width=0.7\textwidth]{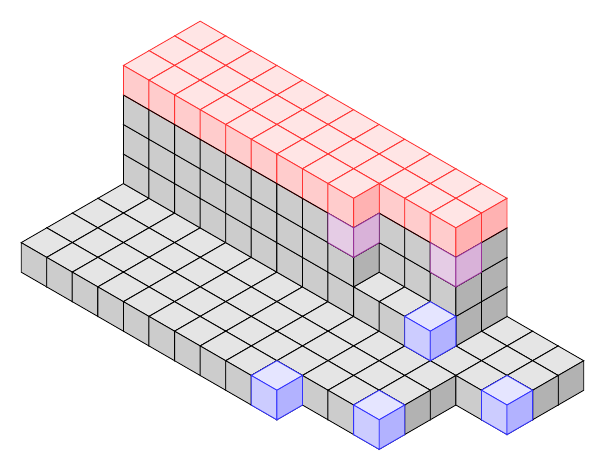}
        \caption{}
        \label{fig_cornersT_primary}
    \end{minipage}
    \begin{minipage}[t]{0.48\textwidth}
        \centering
        \includegraphics[width=0.7\textwidth]{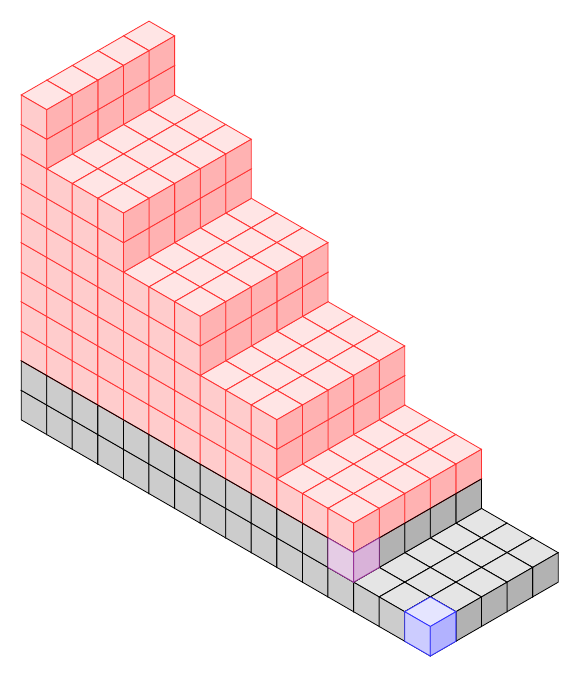}
        \caption{}
        \label{fig_cornersT_secondary}
    \end{minipage}
\end{figure}

\begin{lemma}\label{lemma_type-corners_primary}
    If $S$ is primary, then $2t(S) \geq t' \geq t(S)$.
\end{lemma}
\begin{proof}
If $a_{33} = b_{33}$, then $t' = t(S)$, so suppose that $a_{33} < b_{33}$. Since $t' = | \mathcal{T}| + t(S)$, it follows immediately that $t' \geq  t(S)$, so we only need to prove $t' \leq 2t(S) $, which is equivalent to $|\mathcal{T}| \leq t(S)$.

Consider any $[[\lambda_1,\lambda_2,\lambda_3]] \in \mathcal{T}$. Since $[[\lambda_1,\lambda_2,\lambda_3]]$ is not a corner of $R$ and $T = R \cap \{z \leq a_{33}\} $, we have $\lambda_3 = a_{33}-1$ and $[[\lambda_1,\lambda_2,a_{33}]] \in R$. There exists $k \in \mathbb{N}$ such that $[[\lambda_1,\lambda_2,a_{33}+k]] \in R$ and $[[\lambda_1,\lambda_2,a_{33}+k+1]] \notin R$. Note that $[[\lambda_1,\lambda_2,a_{33}+k]]$ is a corner of $R$, since $T = R \cap \{z \leq a_{33}\}$ and $[[\lambda_1,\lambda_2,a_{33}-1]]$ is a corner of $T$. 

If  $| \mathcal{T}| > t(S)$, then there exist two distinct cubes $[[\lambda_1,\lambda_2,a_{33}-1]], [[\lambda_1',\lambda_2',a_{33}-1]] \in \mathcal{T}$ such that, for some $k, k' \in \mathbb{N}$, $[[\lambda_1,\lambda_2,a_{33}+k]]$ and $[[\lambda_1',\lambda_2',a_{33}+k']]$ are corners of $R$ with the same label. Without loss of generality, suppose $k \geq k'$, $\lambda_1 > \lambda_1'$ and $\lambda_2 < \lambda_2'$ (if $\lambda_1 \geq \lambda_1'$ and $\lambda_2 \geq \lambda_2'$, we would get a contradiction with $[[\lambda_1,\lambda_2,a_{33}-1]]$ and $[[\lambda_1',\lambda_2',a_{33}-1]]$ being distinct corners of $T$). This gives
\begin{equation}
    (\lambda_1 - \lambda_1')d_1 + (k-k')d_3 = (\lambda_2'-\lambda_2)d_2
\end{equation}
which implies $\lambda_2' \geq (\lambda_2'-\lambda_2) \geq a_{22} = b_{22}$, contradicting $[[\lambda_1',\lambda_2',a_{33}-1]] \in T$. Thus $|\mathcal{T}| \leq t(S)$.    
\end{proof}

\begin{lemma}\label{lemma_type-corners_secondary}
    If $S$ is secondary, then $t(S)+1 \geq t' \geq  t(S)$.
\end{lemma}

\begin{proof}
    As before, $t' = | \mathcal{T}| + t(S)$ gives $t' \geq  t(S)$, so it remains to show $t' \leq t(S)+1 $, which is equivalent to $|\mathcal{T}| \leq 1$.

    Consider any $[[\lambda_1,\lambda_2,\lambda_3]] \in \mathcal{T}$. Since $[[\lambda_1,\lambda_2,\lambda_3]]$ is not a corner of $R$ and $T = R \cap \{z \leq a_{33}\} $, we have $\lambda_3 = a_{33}-1$ and $[[\lambda_1,\lambda_2,a_{33}]] \in R$.
    
    If $|\mathcal{T}| \geq 2$, then there exist $[[\lambda_1,\lambda_2,a_{33}-1]],  [[\lambda_1',\lambda_2',a_{33}-1]] \in \mathcal{T}$ with $\lambda_1 > \lambda_1'$ and $\lambda_2 < \lambda_2'$. Additionally, assume that no cube $[[\lambda_1'',\lambda_2'',a_{33}-1]] \in \mathcal{T}$ satisfies $\lambda_1 > \lambda_1'' > \lambda_1'$ and $\lambda_2 < \lambda_2'' < \lambda_2'$. This assumption is equivalent to $(\lambda_1'+1,\lambda_2+1,\lambda) \in V$ for some $\lambda \in \mathbb{Z}$. 
    
    If $\lambda \geq 0$, then by Lemma \ref{lemma_point0connection}, as $\lambda_1'+1,\lambda_2+1 > 0$, we get $(0,\lambda_1'+1,\lambda_2+1,\lambda) \in \varphi^{-1}(a_{00}d_0)$. As $S$ is secondary, we have $\varphi^{-1}(a_{00}d_0) = \{(a_{00},0,0,0),(0,a_{11},0,0)\}$, which is a contradiction. 
    
    Hence $\lambda < 0$, which implies that no cube in the region $\{x \geq \lambda_1'+1, y \geq  \lambda_2+1, z \geq 0\}$ is in $R$. We have $[[\lambda_1,\lambda_2,a_{33}]] \in R$, since $[[\lambda_1,\lambda_2,a_{33}-1]] \in \mathcal{T}$. By Theorem \ref{theoremprocedure}, this gives $[[\lambda_1,\lambda_2+a_{22},0]] \in R$, as $a_{22}d_2 = a_{33}d_3$. However, $[[\lambda_1,\lambda_2+a_{22},0]] \in \{x \geq \lambda_1'+1, y \geq  \lambda_2+1, z \geq 0\}$, as $\lambda_1 > \lambda_1'$, which is a contradiction. Thus $|\mathcal{T}| \leq 1$.   
\end{proof}

Let $F$ be a cube figure.
A cube $[[a,b,0]] \in F$ with $[[a+1,b,0]], [[a,b+1,0]] \notin F$ is called a \emph{stair cube}. Similarly, $[[a,0,c]] \in F$ with $[[a+1,0,c]], [[a,0,c+1]] \notin F$, and $[[0,b,c]] \in F$ with $[[0,b+1,c]], [[0,b,c+1]] \notin F$ are also \emph{stair cubes}.  

Let $\mathcal{S}$ denote the set of all stair cubes of $T$. For the primary semigroup $\langle 784, 841,900,961 \rangle$, Figure \ref{fig_staircubes} shows the corresponding figure $T$ with all stair cubes shown in blue.

\begin{figure}[h]
    \centering
    \includegraphics[width=0.5\textwidth]{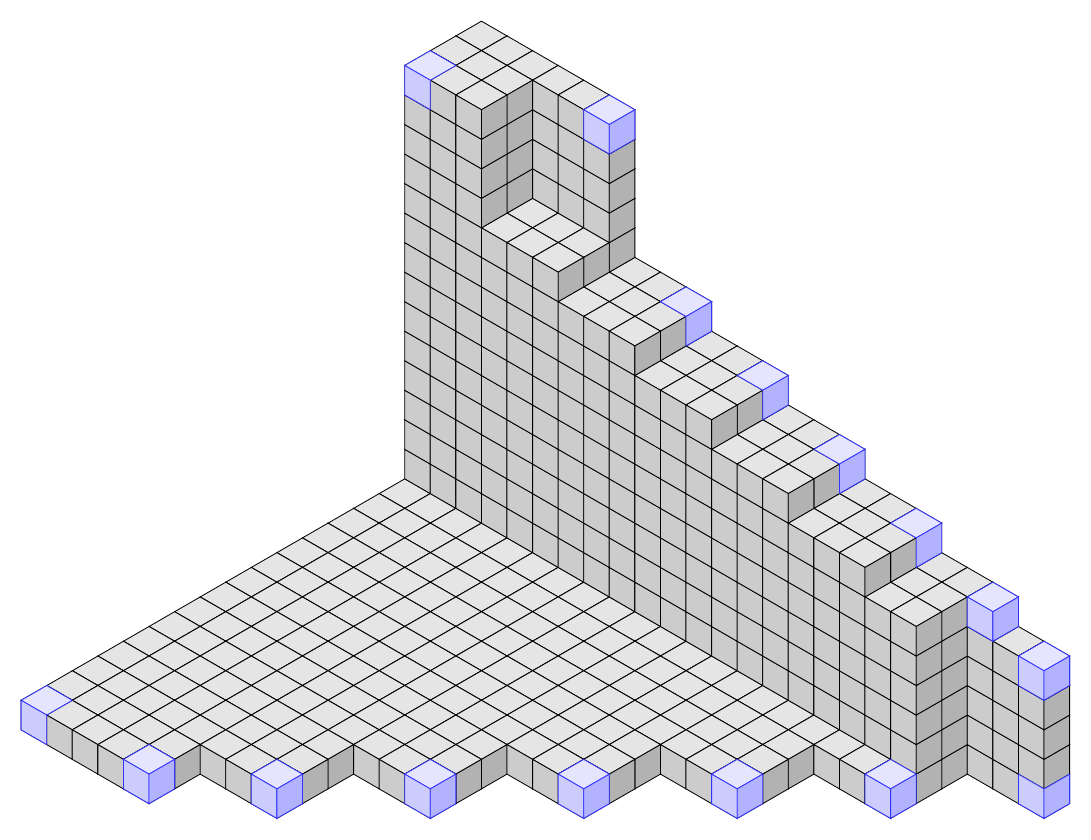}
    \caption{}
    \label{fig_staircubes}
\end{figure}

\begin{lemma}\label{lemma_type_suppthm_stair}
 If there exists $(0,a_{01},a_{02},a_{03}) \in \varphi^{-1}(a_{00}d_0)$ such that $(a_{01},a_{02},a_{03}) \in V$ and exactly one among $a_{01}, a_{02}, a_{03}$ is zero, then $|V^*|+4 \geq |\mathcal{S}| \geq |V^*|+1$. Otherwise $|V^*|+3 \geq |\mathcal{S}| \geq |V^*|$.
\end{lemma}

\begin{proof}
     Order the elements of $\mathcal{S}$ with $z=0$ according to increasing $x$-coordinate. Note that they are also ordered by decreasing $y$-coordinate. In this order, consider two consecutive stair cubes with $z=0$: $[[\lambda_1,\lambda_2,0]], [[\mu_1,\mu_2,0]]$ with $\lambda_1 > \mu_1$ and $\lambda_2 < \mu_2$. Since they are consecutive, there exists $\nu \in \mathbb{N}$ such that $(\mu_1+1,\lambda_2+1,-\nu) \in V$. 

     If no $(0,a_{01},a_{02},a_{03}) \in \varphi^{-1}(a_{00}d_0)$ satisfies $(a_{01},a_{02},a_{03}) \in V$ where exactly one among $a_{01}, a_{02}, a_{03}$ is zero, then $\nu > 0$, since otherwise $(\mu_1+1,\lambda_2+1,0) \in V$ would be such a point by Lemma \ref{lemma_point0connection}. This gives $(\mu_1+1,\lambda_2+1,-\nu) \in V^*$, which we assign to the stair cube $[[\lambda_1,\lambda_2,0]]$. Such a point $(\mu_1+1,\lambda_2+1,-\nu) \in V^*$ exists for all stair cubes in $T \cap \{z = 0\}$ except for the minimal one in our ordering. Moreover, every such point is unique (when $V$ is fixed), and it is clear that we assign exactly one stair cube of $T \cap \{z = 0\}$ to every point of $V^*$ in the region $\{x > 0, y > 0, z < 0\}$.   
     Therefore in the region $\{x > 0, y > 0, z < 0\}$ there is one fewer element of $V^*$ than there are stair cubes in $T \cap \{z = 0\}$. Applying the same approach to stair cubes of $T \cap \{x = 0\}$ and $T \cap \{y = 0\}$, we obtain $|V^*|+3 \geq |\mathcal{S}| \geq |V^*|$, since there are at most three stair cubes of $T$ that belong to more than one of the sets $T \cap \{x = 0\}$, $T \cap \{y = 0\}$, and $T \cap \{z = 0\}$; if such stair cubes exist, they must be among $[[b_{11}-1,0,0]], [[0,b_{22}-1,0]], [[0,0,a_{33}-1]]$.

     If there exists $(0,a_{01},a_{02},a_{03}) \in \varphi^{-1}(a_{00}d_0)$ such that $(a_{01},a_{02},a_{03}) \in V$ and exactly one among $a_{01}, a_{02}, a_{03}$ is zero, then this points in unique by Lemma \ref{lemma_point0uniqueness}. Without loss of generality, suppose $a_{03} = 0$. If $\nu = 0$, then necessarily $(a_{01},a_{02},a_{03}) = (\mu_1+1,\lambda_2+1,0) \in V$. Since $(a_{01},a_{02},a_{03})$ is unique, we get that in the region $\{x > 0, y > 0, z < 0\}$ there are two fewer elements of $V^*$ than there are stair cubes in $T \cap \{z = 0\}$. Again, since $(a_{01},a_{02},a_{03})$ is unique, applying the same approach to stair cubes of $T \cap \{x = 0\}$ and $T \cap \{y = 0\}$ gives the same result as in the previous paragraph. Therefore $|V^*|+4 \geq |\mathcal{S}| \geq |V^*|+1$.  
\end{proof}

\begin{lemma}\label{lemma_type_suppthm_upperbound}
    $  2t'+1 \geq |\mathcal{S}|$.
\end{lemma}

\begin{proof}
    Let $[[a,b,c]]$ be a corner of $T$. Consider how many cubes among $[[0,b,c]],[[a,0,c]],[[a,b,0]]$ are stair cubes of $T$. We claim that there is at most one corner of $T$ that has all three as stair cubes. Suppose to the contrary, that $[[a,b,c]]$ and $[[a',b',c']]$ are two such distinct corners of $T$. Without loss of generality, suppose $a > a'$, $b \geq b'$, $c < c'$. Now, since $[[a',b',0]]$ is a stair cube, we have $[[a'+1,b',0]] \notin T$. This implies $[[a,b,0]] \notin T$, which is a contradiction. 
    
    Therefore, assigning to each corner $[[a,b,c]]$ of $T$ the stair cubes among $[[0,b,c]],[[a,0,c]],$ $[[a,b,0]]$, every corner of $T$ receives at most two stair cubes, with the exception of at most one corner, which receives three. Moreover, each stair cube is assigned to exactly one corner of $T$, hence $ 2t'+1 \geq |\mathcal{S}|$.
\end{proof}

A corner $[[a,b,c]]$ of a cube figure $F$ is an \emph{odd corner}
if none of $[[0,b,c]]$, $[[a,0,c]]$, $[[a,b,0]]$ is a stair cube of $F$.

\begin{prop}\label{prop_type_suppthm_odd_corners}
    Let $F$ be a cube figure. If there is no $[[a,b,c]] \notin F$ such that $[[a-1,b,c]],[[a,b-1,c]],[[a,b,c-1]] \in F$, then $F$ has at most two odd corners.
\end{prop}

\begin{proof}
    Let $[[a,b,c]]$ be an odd corner of $F$. This implies: $[[a+1,b,0]] \in F$ or $[[a,b+1,0]] \in F$, $[[a+1,0,c]] \in F$ or $[[a,0,c+1]] \in F$, and $[[0,b+1,c]] \in F$ or $[[0,b,c+1]] \in F$. Suppose that $[[a+1,b,0]]$ and $[[a,b+1,0]]$ are both in $F$ --- we claim that this cannot be the case. 

    If $[[a+1,0,c]] \in F$, then there exists $b \geq \lambda > 0$ such that $[[a+1,\lambda-1,c]] \in F$ and $[[a+1,\lambda,c]] \notin F$. Then, there exists $c \geq \mu > 0$ such that $[[a+1,\lambda,\mu-1]] \in F$ and $[[a+1,\lambda,\mu]] \notin F$, as $[[a+1,\lambda,0]] \leq [[a+1,b,0]] \in F$. However, since $[[a+1,\lambda,\mu]] \notin F$ and $[[a,\lambda,\mu]], [[a+1,\lambda-1,\mu]], [[a+1,\lambda,\mu-1]] \in F$, we get a contradiction with the hypothesis. On Figure \ref{fig_type_theorem_1}, we can see an example of such a situation, where $P = (a+1,\lambda,\mu)$. 
    
    If $[[0,b+1,c]] \in F$, we get an analogous contradiction; thus, since $[[a,b,c]]$ is an odd corner and neither $[[a+1,0,c]]$ nor $[[0,b+1,c]]$ is in $F$, we have $[[a,0,c+1]], [[0,b,c+1]] \in F$. Now, there exist $a \geq \lambda > 0$, $b \geq \mu > 0$, such that $[[\lambda,\mu,c+1]] \notin F$ and $[[\lambda-1,\mu,c+1]], [[\lambda,\mu-1,c+1]] \in F$. Since $[[\lambda,\mu,c]] \in F$ also, we get a contradiction with the hypothesis. On Figure \ref{fig_type_theorem_2}, we can see an example of such a situation, where $P = (\lambda,\mu,c+1)$.

\begin{figure}[h]
    \centering
    \begin{minipage}[t]{0.48\textwidth}
        \centering
        \includegraphics[width=0.7\textwidth]{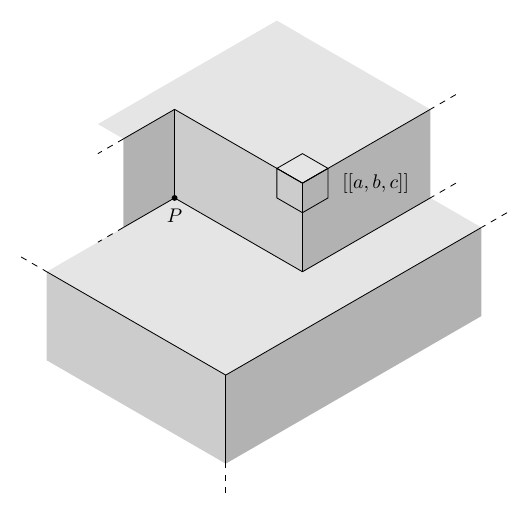}
        \caption{}
        \label{fig_type_theorem_1}
    \end{minipage}
    \begin{minipage}[t]{0.48\textwidth}
        \centering
        \includegraphics[width=0.7\textwidth]{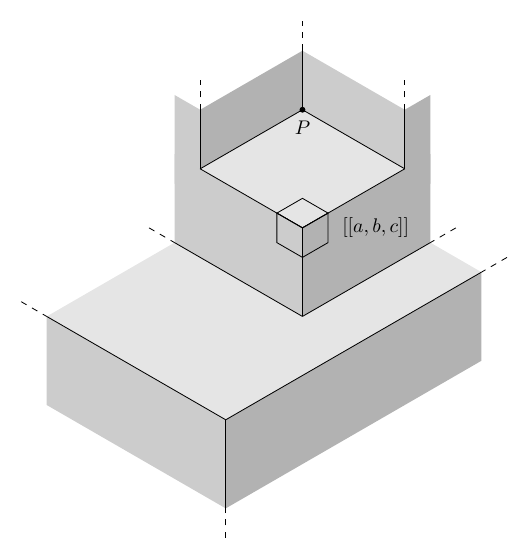}
        \caption{}
        \label{fig_type_theorem_2}
    \end{minipage}
\end{figure}

    Similarly, we obtain that exactly one of $[[a+1,0,c]]$ and $[[a,0,c+1]]$ is in $F$, and that exactly one of $[[0,b+1,c]]$ and $[[0,b,c+1]]$ is in $F$. 

    Now, we claim that any two distinct odd corners $[[a,b,c]]$ and $[[a',b',c']]$ of $F$ must agree in at least one coordinate. Otherwise, without loss of generality, suppose that $a < a'$, $b < b'$, $c > c'$ (we cannot have $[[a,b,c]] < [[a',b',c']]$, as they are both corners). Then, clearly $[[a+1,b,0]], [[a,b+1,0]] \in F$, since $[[a+1,b,0]], [[a,b+1,0]] < [[a',b',c']] \in F$, which contradicts the claim proved at the beginning.

    Thus, without loss of generality, suppose that $a < a'$, $b = b'$, $c > c'$.  Suppose that $[[a'',b'',c'']]$ is a third distinct odd corner of $F$. If $b'' \neq b$, then from the previous paragraph
    we obtain $[[a'',b'',c'']] = [[a,b'',c']]$ or $[[a'',b'',c'']] = [[a',b'',c]]$ --- recall that $a < a'$ and $c > c'$, so we cannot have $a'' = a' = a$ or $c'' = c' = c$.

    Suppose $[[a'',b'',c'']] = [[a',b'',c]]$. If $b'' \geq b' = b$, then $[[a',b'',c]] >[[a',b',c']]$, which contradicts the fact that $[[a',b',c']]$ is a corner of $F$. Hence $ b'' < b' = b$. Now, there exist $a' \geq \lambda > a$, $b \geq \mu > b''$, and $c \geq \nu > c'$, such that $[[\lambda,\mu,\nu]] \notin F$ and $[[\lambda-1,\mu,\nu]],[[\lambda,\mu-1,\nu]],[[\lambda,\mu,\nu-1]] \in F$, which contradicts the hypothesis. On Figure \ref{fig_type_theorem_3}, we can see an example of such a situation, where $P = (\lambda,\mu,\nu)$. 
    
    Suppose $[[a'',b'',c'']] = [[a,b'',c']]$. If $ b'' \leq b'=b $, then $[[a,b'',c']] < [[a,b,c]]$, which contradicts the fact that $ [[a,b'',c']]$ is a corner of $F$. Hence $b'' > b' =b$. Now, since $[[a+1,b,0]] \leq [[a',b',0]] \in F$ and $ [[a,b+1,0]] \leq [[a,b'',0]] = [[a'',b'',0]] \in F$, we have $[[a+1,b,0]], [[a,b+1,0]] \in F$, which contradicts the claim proved at the beginning. On Figure \ref{fig_type_theorem_4}, we can see an example of such a situation. 

\begin{figure}[h]
    \centering
    \begin{minipage}[t]{0.48\textwidth}
        \centering
        \includegraphics[width=0.7\textwidth]{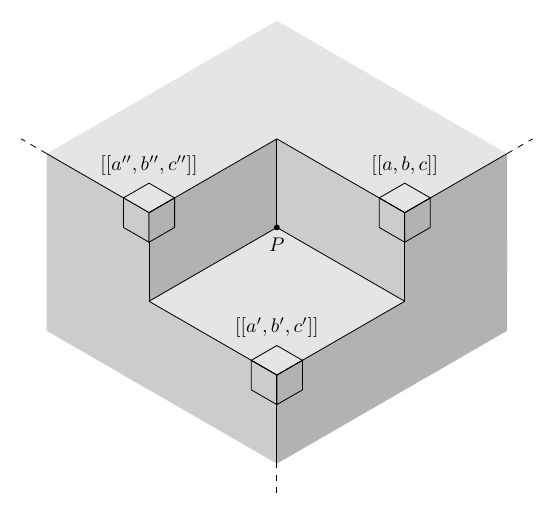}
        \caption{}
        \label{fig_type_theorem_3}
    \end{minipage}
    \begin{minipage}[t]{0.48\textwidth}
        \centering
        \includegraphics[width=0.7\textwidth]{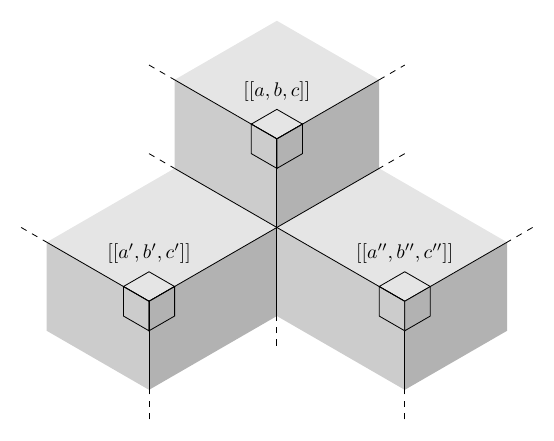}
        \caption{}
        \label{fig_type_theorem_4}
    \end{minipage}
\end{figure}

    Hence $b= b' = b'' $, and without loss of generality, suppose that
    $a < a' < a''$ and $c > c' > c''$. We have $[[a'+1,0,c']] \in F$ or $[[a',0,c'+1]] \in F$, as $[[a',b',c']]$ is an odd corner. 
    
    If $[[a',0,c'+1]] \in F$, then there exist $a' \geq \lambda > a$ and $b \geq \mu > 0$ such that $[[\lambda,\mu,c'+1]] \notin F$ and $[[\lambda-1,\mu,c'+1]], [[\lambda,\mu-1,c'+1]], [[\lambda,\mu,c']] \in F$, which contradicts the hypothesis. On Figure \ref{fig_type_theorem_5}, we can see an example of such a situation, where $P = (\lambda,\mu,c'+1)$.

     If $[[a'+1,0,c']] \in F$, then there exist $a'' \geq \lambda > a'$ and $b \geq \mu > 0$ such that $[[\lambda,\mu,c''+1]] \notin F$ and $[[\lambda-1,\mu,c''+1]], [[\lambda,\mu-1,c''+1]], [[\lambda,\mu,c'']] \in F$, which contradicts the hypothesis. On Figure \ref{fig_type_theorem_6}, we can see an example of such a situation, where $P = (\lambda,\mu,c'+1)$.  \qedhere

\begin{figure}[h]
    \centering
    \begin{minipage}[t]{0.48\textwidth}
        \centering
        \includegraphics[width=0.7\textwidth]{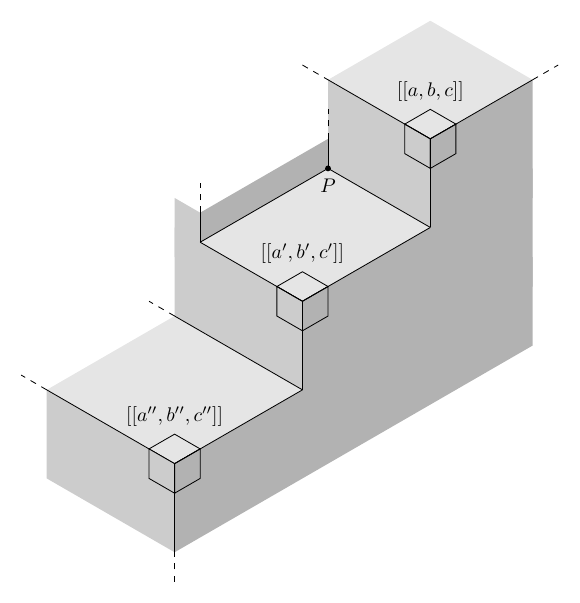}
        \caption{}
        \label{fig_type_theorem_5}
    \end{minipage}
    \begin{minipage}[t]{0.48\textwidth}
        \centering
        \includegraphics[width=0.7\textwidth]{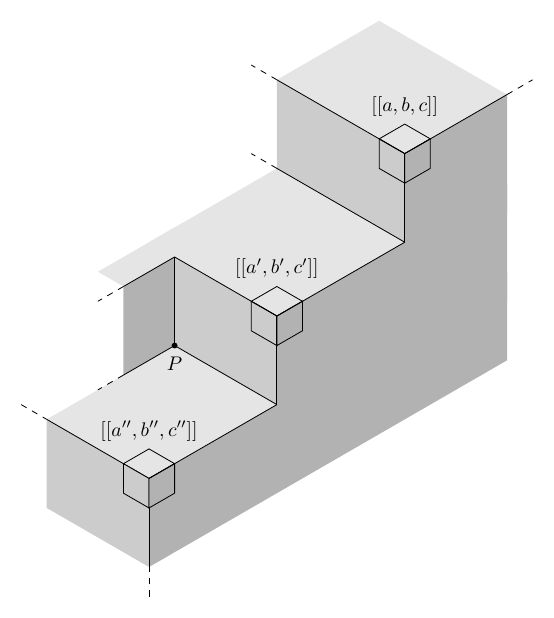}
        \caption{}
        \label{fig_type_theorem_6}
    \end{minipage}
\end{figure}
\end{proof}

\begin{lemma}\label{lemma_type_suppthm_main_lower}
 If there exists $(0,a_{01},a_{02},a_{03}) \in \varphi^{-1}(a_{00}d_0)$ such that $(a_{01},a_{02},a_{03}) \in V$ and $a_{01} a_{02} a_{03}$ $ > 0$, then $|\mathcal{S}| \geq t'-8$. Otherwise $|\mathcal{S}| \geq t'-2$.
\end{lemma}

\begin{proof}

    If no $(0,a_{01},a_{02},a_{03}) \in \varphi^{-1}(a_{00}d_0)$ has $(a_{01},a_{02},a_{03}) \in V$ and $a_{01}a_{02}a_{03} > 0$, then $T$ satisfies the assumptions of Proposition \ref{prop_type_suppthm_odd_corners}, hence it has at most two odd corners. Then each non-odd corner $[[a,b,c]]$ of $T$ has at least one associated stair cube, chosen from $[[0,b,c]], [[a,0,c]],[[a,b,0]]$. Every stair cube is associated to at most one corner of $T$, hence $|\mathcal{S}| \geq t' -2$.

    Suppose that there exists $(0,a_{01},a_{02},a_{03}) \in \varphi^{-1}(a_{00}d_0)$ with $(a_{01},a_{02},a_{03}) \in V$ and $a_{01}a_{02}a_{03} > 0$. By Lemmas \ref{lemma_point0uniqueness} and \ref{lemma_point0connection} we know that this point is unique. Let $T'$ be the figure obtained by removing from the initial collection of cubes all the regions associated to each of $V \setminus \{ (a_{01},a_{02},a_{03})\}$. Note that the sets of stair cubes of $T$ and $T'$ are equal. Moreover, $T'$ satisfies the assumptions of Proposition \ref{prop_type_suppthm_odd_corners}, hence it has at most two odd corners.

    Now, deleting from $T'$ the region associated to $(a_{01},a_{02},a_{03})$ will yield $T$. We claim that this operation can add at most six new odd corners. Note that any new odd corner could not have been a corner of $T'$ --- by the definition, a corner that is not odd in $T'$ cannot become an odd corner merely because cubes were deleted. Thus every new odd corner must touch the region associated to $(a_{01},a_{02},a_{03})$ with one of its faces.

    Assume, for contradiction, that seven new odd corners appear. This means that three of them touch the same side of the region associated to $(a_{01},a_{02},a_{03})$. Without loss of generality, suppose that $[[x_1,y_1,a_{03}-1]], [[x_2,y_2,a_{03}-1]], [[x_3,y_3,a_{03}-1]]$ are three distinct odd corners, where $x_1 > x_2 > x_3 \geq a_{01}$ and $y_3 > y_2 > y_1 \geq a_{02}$. We have $[[x_2+1,y_2,0]] \in T$ or $[[x_2,y_2+1,0]] \in T$. If $[[x_2+1,y_2,0]] \in T$, then there exists $x_1 \geq \lambda > x_2$, $y_2 \geq  \mu > y_1$, and $a_{03}-1 \geq \nu > 0$, such that $[[\lambda,\mu,\nu]] \notin F$ and $[[\lambda-1,\mu,\nu]],[[\lambda,\mu-1,\nu]],[[\lambda,\mu,\nu-1]] \in F$. This implies $(\lambda,\mu,\nu) \in V$, which by Lemma \ref{lemma_point0connection}, contradicts Lemma \ref{lemma_point0uniqueness}. On Figure \ref{fig_type_theorem_7}, we can see an example of such a situation, where $P = (\lambda,\mu,\nu)$. If $[[x_2,y_2+1,0]] \in T$, we proceed analogously.

    \begin{figure}[h]
    \centering
    \includegraphics[width=0.5\textwidth]{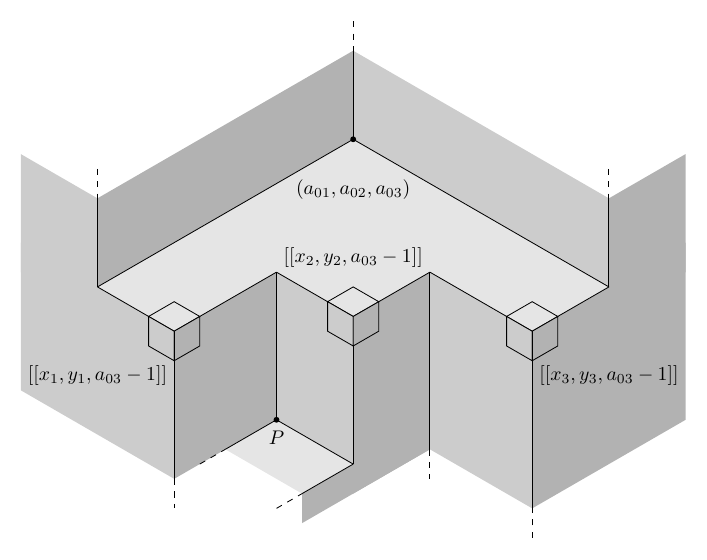}
    \caption{}
    \label{fig_type_theorem_7}
    \end{figure}

    Thus, since $T'$ has at most two odd corners, $T$ can have at most eight odd corners.
    Each non-odd corner $[[a,b,c]]$ of $T$ has at least one associated stair cube, chosen from $[[0,b,c]], [[a,0,c]],[[a,b,0]]$. Every stair cube is associated to at most one corner of $T$, hence $|\mathcal{S}| \geq t' - 8$. 
\end{proof}

\begin{thm}\label{theorem_type1}
    If $S$ is primary, then $4 t(S) +1 \geq |V^*| \geq t(S)-11$.  If $S$ is secondary, then $2 t(S) +3 \geq |V^*| \geq t(S) - 5$.
\end{thm}

\begin{proof}
    Suppose that $S$ is primary. The upper bound for $|V^*|$ follows from the chain of inequalities:
    \begin{equation}
        |V^*| \leq |\mathcal{S}| \leq 2 t' + 1 \leq 4 t(S) +1
    \end{equation}
    which are the result of Lemmas \ref{lemma_type_suppthm_stair}, \ref{lemma_type_suppthm_upperbound}, and \ref{lemma_type-corners_primary}, respectively. To prove the lower bound for $|V^*|$ we need to split into two cases.

    Suppose that there exists $(0,a_{01},a_{02},a_{03}) \in \varphi^{-1}(a_{00}d_0)$ such that $(a_{01},a_{02},a_{03}) \in V$ and $a_{01} a_{02} a_{03} > 0$. Then, by Lemma \ref{lemma_point0uniqueness}, the point $(a_{01},a_{02},a_{03}) \in V$ is unique and there does not exist $(0,a_{01}',a_{02}',a_{03}') \in \varphi^{-1}(a_{00}d_0)$ such that $(a_{01}',a_{02}',a_{03}') \in V$ and exactly one among $a_{01}', a_{02}', a_{03}'$ is zero. Hence, we have the following chain of inequalities: 
    \begin{equation}
        |V^*| \geq |\mathcal{S}|-3 \geq t' - 11 \geq t(S) - 11
    \end{equation}
    which are the result of Lemmas \ref{lemma_type_suppthm_stair}, \ref{lemma_type_suppthm_main_lower}, and \ref{lemma_type-corners_primary}, respectively.

    Suppose that there does not exist $(0,a_{01},a_{02},a_{03}) \in \varphi^{-1}(a_{00}d_0)$ such that $(a_{01},a_{02},a_{03}) \in V$ and $a_{01} a_{02} a_{03} > 0$. Then, we have the following chain of inequalities: 
    \begin{equation}
        |V^*| \geq |\mathcal{S}|-4 \geq t' - 6 \geq t(S) - 6
    \end{equation}
    which are the result of Lemmas \ref{lemma_type_suppthm_stair}, \ref{lemma_type_suppthm_main_lower}, and \ref{lemma_type-corners_primary}, respectively. As $t(S) - 6 \geq t(S) - 11$, the first claim follows.

    Suppose that $S$ is secondary. The upper bound for $|V^*|$ is due to the following chain of inequalities:
    \begin{equation}
        |V^*| \leq |\mathcal{S}| \leq 2 t' + 1 \leq 2 t(S) +3
    \end{equation}
    which are the result of Lemmas \ref{lemma_type_suppthm_stair}, \ref{lemma_type_suppthm_upperbound}, and \ref{lemma_type-corners_secondary}, respectively. Since $\varphi^{-1}(a_{00}d_0) = \{(a_{00},0,0,0)$, $(0,a_{11},0,0)\}$, there does not exist $(0,a_{01},a_{02},a_{03}) \in \varphi^{-1}(a_{00}d_0)$ such that $a_{01} a_{02} a_{03} > 0$, nor such that exactly one among $a_{01}, a_{02}, a_{03}$ is zero. Hence, we have the following chain of inequalities: 
    \begin{equation}
        |V^*| \geq |\mathcal{S}|-3 \geq t' -5 \geq t(S) - 5
    \end{equation} 
    which are the result of Lemmas \ref{lemma_type_suppthm_stair}, \ref{lemma_type_suppthm_main_lower}, and \ref{lemma_type-corners_secondary}, respectively.
\end{proof}

\subsection{$V$ and $\eta(S)$}
To relate $|V^*|$ with $\eta(S)$, we need to know how many elements of $V^*$ do not correspond to Betti elements. 

Let
$$ \mathcal{N} = \{ (x,y,z) \in V^* \mid  \max\{ x,0\} d_1 + \max\{y,0\}d_2 + \max\{z,0\}d_3 \notin \textup{Betti}(S)  \}.$$

\begin{lemma}\label{lemma_bettiT_primary}
     If $S$ is primary, then $|\mathcal{N}| \leq 3$.
\end{lemma}
\begin{proof}

In this proof, as in the proof of Lemma \ref{lemmabettifind_primary}, it will be more convenient to drop the assumption that $S$ is ordered --- we only suppose that $a_{ii}=b_{ii}$ for at least two distinct $i \in \{1,2,3\}$ and that $T = R \cap \{x \leq a_{11},y\leq a_{22}, z \leq a_{33}\}$. 

Suppose $|\mathcal{N}| \geq 4$. Then, we can find two distinct points in $\mathcal{N}$ whose negative coordinates occur in the same position. Without loss of generality, let $(-\lambda_1,\lambda_2,\lambda_3)$ and $(-\lambda_1',\lambda_2',\lambda_3')$ be distinct points in $\mathcal{N}$ (where $\lambda_i>0$). By Lemma \ref{lemma_catenaryT_primary}, we have 
\begin{align}
    a_{00}d_0 + \lambda_1d_1 &=  \lambda_2d_2 + \lambda_3d_3,  \label{lemma_bettiT_primary_eq3} \\
    a_{00}d_0 + \lambda_1'd_1 &=  \lambda_2'd_2 + \lambda_3'd_3, \label{lemma_bettiT_primary_eq4}
\end{align}
where $ \lambda_i,\lambda_i' < a_{ii}$ for $i = 1,2,3$. Subtracting \eqref{lemma_bettiT_primary_eq3} from \eqref{lemma_bettiT_primary_eq4} gives 
\begin{equation}
    (\lambda_1-\lambda_1')d_1 + (\lambda_2'-\lambda_2)d_2 + (\lambda_3'-\lambda_3)d_3 = 0.
\end{equation}
Since $(-\lambda_1,\lambda_2,\lambda_3) \neq (-\lambda_1',\lambda_2',\lambda_3')$, not all coefficients above are zero. Without loss of generality, suppose $(\lambda_1-\lambda_1') > 0$, $(\lambda_2'-\lambda_2) \leq 0$, $(\lambda_3'-\lambda_3) \leq 0$. This implies $\lambda_1 \geq (\lambda_1-\lambda_1') \geq a_{11}$, which is a contradiction. 
\end{proof}

\begin{lemma}\label{lemma_bettiT_secondary}
     If $S$ is secondary, then $|\mathcal{N}| \leq 1$.
\end{lemma}
\begin{proof} 
By Lemmas \ref{lemma_catenary_lambda1_secondary} and \ref{lemma_catenaryT_secondary_case1}, no points of the form $(-\lambda_1,\lambda_2,\lambda_3)$ or $(\lambda_1,-\lambda_2,\lambda_3)$ (where $\lambda_i \in \mathbb{Z}_{>0}$) are in $\mathcal{N}$.

Thus every point in $\mathcal{N}$ is of the form $(\lambda_1,\lambda_2,-\lambda_3)$ (where $\lambda_i \in \mathbb{Z}_{>0}$). 
If $|\mathcal{N}| \geq 2$, then let $(\lambda_1,\lambda_2,-\lambda_3)$ and $(\lambda_1',\lambda_2',-\lambda_3')$ be distinct points in $\mathcal{N}$. By Lemma \ref{lemma_catenaryT_secondary_case2}, we have
\begin{align}
    a_{00}d_0 + \lambda_3d_3 &=  \lambda_1d_1 + \lambda_2d_2,  \label{lemma_bettiT_secondary_eq1} \\
    a_{00}d_0 + \lambda_3'd_3 &=  \lambda_1'd_1 + \lambda_2'd_2, \label{lemma_bettiT_secondary_eq2}
\end{align}
where $\lambda_3, \lambda_3' < a_{33}$ Since $T = R \cap \{z \leq a_{33}\}$, we have $\lambda_1, \lambda_1' < b_{11}=a_{11}$.  

Without loss of generality, suppose $\lambda_3 \geq \lambda_3'$. Subtracting \eqref{lemma_bettiT_secondary_eq2} and \eqref{lemma_bettiT_secondary_eq1}, we obtain
\begin{equation}
    (\lambda_3- \lambda_3')d_3 + (\lambda_1'-\lambda_1)d_1 =  (\lambda_2-\lambda_2')d_2
\end{equation}
where all above coefficients are nonnegative, as otherwise we would get a contradiction with $\lambda_3 < a_{33}$ or $\lambda_1 < a_{11}$. As $(\lambda_1,\lambda_2,-\lambda_3) \neq (\lambda_1',\lambda_2',-\lambda_3')$, we obtain $\lambda_2 >
(\lambda_2-\lambda_2') \geq a_{22}$. We have 
\begin{equation}
     a_{00}d_0 =  \lambda_1d_1 + (\lambda_2-a_{22})d_2 + (a_{33}-\lambda_3)d_3
\end{equation}
where all above coefficients are positive. This contradicts $\varphi^{-1}(a_{00}d_0) = \{(a_{00},0,0,0), (0,a_{11},0,0)\}$. 
\end{proof}

Let
$$U_0 = \{ \max\{ x,0\} d_1 + \max\{y,0\}d_2 + \max\{z,0\}d_3 \mid (x,y,z) \in V^* \setminus \mathcal{N} \}.$$
Note that $U_0$ is independent of the choice of $V$.

\begin{lemma}\label{lemma_U0_VN}
     $|U_0| = |V^* \setminus \mathcal{N}|$.
\end{lemma}

\begin{proof}
    Suppose that there exist distinct $(-\lambda_1,\lambda_2,\lambda_3)$ and $(-\lambda_1',\lambda_2',\lambda_3') \in V^* \setminus \mathcal{N}$ (where $\lambda_i > 0$) such that $\lambda_2d_2 + \lambda_3d_3 = \lambda_2'd_2 + \lambda_3'd_3$. Without loss of generality, suppose $\lambda_2 > \lambda_2'$ and $\lambda_3 < \lambda_3'$. Now, as $(\lambda_3'-\lambda_3)d_3 = (\lambda_2-\lambda_2')d_2 > 0$, we obtain $\lambda_3' > (\lambda_3'-\lambda_3) \geq a_{33}$. This contradicts $(-\lambda_1',\lambda_2',\lambda_3') \in V^*$, since $T = R \cap \{z \leq a_{33}\}$. 
    
   If there exist distinct $(\lambda_1,-\lambda_2,\lambda_3)$ and $ (\lambda_1',-\lambda_2',\lambda_3') \in V^* \setminus \mathcal{N}$ (where $\lambda_i > 0$) such that $\lambda_1d_1 + \lambda_3d_3 = \lambda_1'd_1 + \lambda_3'd_3$, we proceed analogously as above and derive a contradiction. 
    
   Suppose that there exist distinct $(\lambda_1,\lambda_2,-\lambda_3)$ and $ (\lambda_1',\lambda_2',-\lambda_3') \in V^* \setminus \mathcal{N}$ (where $\lambda_i > 0$) such that $\lambda_1d_1 + \lambda_2d_2 = \lambda_1'd_1 + \lambda_2'd_2$. Without loss of generality, suppose $\lambda_1 > \lambda_1'$ and $\lambda_2 < \lambda_2'$, which gives $\lambda_1 > (\lambda_1 - \lambda_1') \geq a_{11} = b_{11}$ (as $S$ is ordered). This contradicts $(\lambda_1,\lambda_2,-\lambda_3) \in V^*$, since $T \subseteq R \subseteq \{x \leq b_{11}, y \leq b_{22}, z \leq b_{33}\}$.

   Suppose that $S$ is primary and there exist $(-\lambda_1,\lambda_2,\lambda_3)$ and $(\lambda_1',-\lambda_2',\lambda_3') \in V^* \setminus \mathcal{N}$ (where $\lambda_i > 0$) such that $\lambda_2d_2 + \lambda_3d_3 = \lambda_1'd_1 + \lambda_3'd_3$. Without loss of generality, suppose $\lambda_3 \geq \lambda_3'$, which gives $\lambda_1' \geq a_{11}$, which contradicts  $(\lambda_1',-\lambda_2',\lambda_3') \in V^*$, since $T \subseteq \{x \leq a_{11}, y \leq a_{22}, z \leq a_{33}\}$. All other cases follow analogously. 

   Suppose that $S$ is secondary and there exist $(-\lambda_1,\lambda_2,\lambda_3)$ and $(\lambda_1',-\lambda_2',\lambda_3') \in V^* \setminus \mathcal{N}$ (where $\lambda_i > 0$) such that $\lambda_2d_2 + \lambda_3d_3 = \lambda_1'd_1 + \lambda_3'd_3$. Since $(\lambda_1',-\lambda_2',\lambda_3') \in V^*$, we have $\lambda_1' < b_{11} = a_{11}$. Thus we have
   \begin{equation}
       \lambda_2d_2 = \lambda_1'd_1 + (\lambda_3'-\lambda_3)d_3
   \end{equation}
   where $(\lambda_3'-\lambda_3) \geq 0$. This gives $\lambda_2 \geq a_{22}$, hence
   \begin{equation}
       (\lambda_2-a_{22})d_2 = \lambda_1'd_1 + (\lambda_3'-\lambda_3-a_{33})d_3
   \end{equation}
   where $(\lambda_2-a_{22}) \geq 0$. Since $a_{11} > \lambda_1' > 0$, we have $(\lambda_3'-\lambda_3-a_{33}) \geq 0$, which gives $(\lambda_2-a_{22}) \geq a_{22}$. This leads to an infinite descent, which is impossible. 
   All other cases follow analogously.
\end{proof}

Let $\rho$ be a minimal presentation of $S$ and let $\rho^* = \{(a,b) \in \mathbb{N}^{4} \times \mathbb{N}^{4}  \mid (a,b) \in \rho, |\textup{supp}(a) | = | \textup{supp}(b) | =2 \} $. To proceed, we require a result of Bresinsky \cite{Bresinsky1988}. 

\begin{lemma}\label{lemma_rho_bresinsky}
   \textup{\cite[Corollary 2]{Bresinsky1988}} There exists a minimal presentation $\rho$ such that $\eta(S) - |\rho^*| \leq 4$.
\end{lemma}

\begin{lemma}\label{lemma_betti_VNrho_primary}
     If $S$ is primary, then $|U_0| + 4 \geq  \eta(S) \geq |U_0| +3 $. 
\end{lemma}
\begin{proof}
    By Lemma \ref{lemma_catenary_minrelations}, and Theorems \ref{theoremprocedure} and \ref{theorembettifind_primary}, we have
    \begin{equation}\label{lemma_betti_VNrho_primary_eq1}
        \textup{Betti}(S) = U_0 \sqcup \{a_{00}d_0, a_{11}d_1, a_{22}d_2, a_{33}d_3\}. 
    \end{equation}
    Lemma \ref{lemma_catenary_minrelations} tells us that $(a_{00},0,0,0)$ is the only element in its $\mathcal{R}$-class. The same is true for $(0,a_{11},0,0)$, $(0,0,a_{22},0)$, and $(0,0,0,a_{33})$, so the sets on the right-hand side
    of \eqref{lemma_betti_VNrho_primary_eq1} are indeed disjoint. 
    
    If there exists a factorization in $\bigcup_{i=0}^3 \varphi^{-1}(a_{ii}d_i)$ that is distinct from the four canonical ones, then $a_{00}d_0, a_{11}d_1, a_{22}d_2, a_{33}d_3$ together contribute at least three elements to $\rho$ by Lemma \ref{lemma_catenary_minrelations} and Theorem \ref{theorem_minimalpresentation}. Otherwise, since $S$ is not secondary, we have $a_{00}d_0= a_{11}d_1= a_{22}d_2= a_{33}d_3$, which has four $\mathcal{R}$-classes by Lemma \ref{lemma_catenary_minrelations}, so it contributes three elements to $\rho$ by Theorem \ref{theorem_minimalpresentation}.

    From Lemmas \ref{lemma_catenary_primary} and \ref{lemma_catenarycase1_primary} we know that every element of $U_0$ has exactly two $\mathcal{R}$-classes, so each contributes exactly one element to $\rho$, by Theorem \ref{theorem_minimalpresentation}. This, combined with the previous paragraph, gives $\eta(S) \geq |U_0| + 3$. 

    Suppose, for contradiction, that $\eta(S) - |U_0| > 4$. By \eqref{lemma_betti_VNrho_primary_eq1}, since $U_0$ contributes exactly $|U_0|$ elements to $\rho$, the elements $a_{00}d_0$, $a_{11}d_1$, $a_{22}d_2$, $a_{33}d_3$ must together contribute more than $4$ elements to $\rho$ --- and by Theorem \ref{theorem_minimalpresentation}, this holds for every minimal presentation $\rho$, since this count is determined only by the number of $\mathcal{R}$-classes. 
    By Lemma \ref{lemma_catenary_minrelations}, there do not exist $z_1,z_2 \in \varphi^{-1}(a_{ii}d_i)$ such that $|\textup{supp}(z_1)| = |\textup{supp}(z_2)| = 2$ and $\textup{supp}(z_1) \cap \textup{supp}(z_2) = \varnothing$. Thus in every minimal presentation $\rho$, there will be more than $4$ elements that are not in $\rho^*$. This gives a contradiction with Lemma \ref{lemma_rho_bresinsky}, therefore $\eta(S) \leq |U_0| + 4 $.
\end{proof}

\begin{lemma}\label{lemma_betti_VNrho_secondary}
    If $S$ is secondary, then $|U_0| + 6 \geq  \eta(S) \geq |U_0| + 2 $. 
\end{lemma}
\begin{proof}
    By Lemma \ref{lemma_catenary_minrelations} and Theorems \ref{theoremprocedure} and \ref{theorembettifind_secondary}, we have
   \begin{equation}\label{lemma_betti_VNrho_secondary_eq1}
        U_0 \sqcup \{a_{00}d_0, a_{22}d_2\} \subseteq \textup{Betti}(S) \subseteq U_0 \cup \{a_{00}d_0, a_{22}d_2, b_{22}d_2, b_{33}d_3\}. 
    \end{equation}
     From Lemmas \ref{lemma_catenary_lambda1_secondary} and \ref{lemma_catenary_lambda23_secondary} we know that every element of $U_0$ has exactly two $\mathcal{R}$-classes, so it contributes exactly one element to $\rho$, by Theorem \ref{theorem_minimalpresentation}. We also have $((a_{00},0,0,0),(0,a_{11},0,0) ) \in \rho$ and $((0,0,a_{22},0),(0,0,0,a_{33})) \in  \rho$, hence $\eta(S) \geq |U_0| + 2$.

     We claim that $b_{22}d_2$ and $b_{33}d_3$ together contribute at most four elements to $\rho$. Otherwise, without loss of generality, suppose that $b_{22}d_2$ contributes at least three elements to $\rho$, so it has at least four $\mathcal{R}$-classes. The only way for this to be possible is $\varphi^{-1}(b_{22}d_2) = \{(b_{20},0,0,0)$, $(0,b_{21},0,0)$, $(0,0,b_{22},0)$, $(0,0,0,b_{23})\}$. Note that $b_{23} \geq b_{33}$. However, if $b_{23} > b_{33}$, then $(b_{30}$, $b_{31}$, $b_{32}$, $(b_{23}-b_{33})) \in \varphi^{-1}(b_{22}d_2)$, which is a contradiction, since $b_{30}, (b_{23}-b_{33}) > 0 $. Hence $b_{23}=b_{33}$ and $b_{22}d_2 = b_{33}d_3$. Therefore, in the current case, $b_{22}d_2$ and $b_{33}d_3$ together contribute at most three elements to $\rho$, by Theorem \ref{theorem_minimalpresentation}, since $b_{22}d_2 = b_{33}d_3$ has at most four $\mathcal{R}$-classes.

    Therefore, by \eqref{lemma_betti_VNrho_secondary_eq1}, we have $\eta(S) \leq |U_0| + 2 + 4 = |U_0| + 6$.   
\end{proof}

\begin{cor}\label{corollary_type_main}
     If $S$ is primary, then $4 t(S) +5 \geq \eta(S) \geq t(S) - 11$.  If $S$ is secondary, then $2 t(S) +9 \geq \eta(S) \geq t(S) - 4$. 
\end{cor}

\begin{proof}
    Suppose that $S$ is primary. The lower bound for $\eta(S)$ follows from the chain of inequalities:
    \begin{equation}
        \eta(S) \geq |U_0| + 3 = |V^* \setminus \mathcal{N}| + 3 \geq |V^*|  \geq t(S) - 11
    \end{equation}
     which are the result of Lemmas \ref{lemma_betti_VNrho_primary}, \ref{lemma_U0_VN}, \ref{lemma_bettiT_primary}, and Theorem \ref{theorem_type1}, respectively. Moreover, the upper bound for $\eta(S)$ is due to the following chain of inequalities:
    \begin{equation}
        \eta(S) \leq |U_0| + 4 = |V^* \setminus \mathcal{N}| + 4 \leq |V^*| +4 \leq 4t(S) +5
    \end{equation}
     which are the result of Lemmas \ref{lemma_betti_VNrho_primary}, \ref{lemma_U0_VN}, and Theorem \ref{theorem_type1}, respectively.

     Suppose that $S$ is secondary. The lower bound for $\eta(S)$ is due to the following chain of inequalities:
    \begin{equation}
        \eta(S) \geq |U_0| + 2 = |V^* \setminus \mathcal{N}| + 2 \geq |V^*| +1 \geq t(S) - 4
    \end{equation}
     which are the result of Lemmas \ref{lemma_betti_VNrho_secondary}, \ref{lemma_U0_VN}, \ref{lemma_bettiT_secondary}, and Theorem \ref{theorem_type1}, respectively. Moreover, the upper bound for $\eta(S)$ is due to the following chain of inequalities:
    \begin{equation}
        \eta(S) \leq |U_0| + 6 = |V^* \setminus \mathcal{N}| + 6 \leq |V^*| +6 \leq 2t(S) +9
    \end{equation}
     which are the result of Lemmas \ref{lemma_betti_VNrho_secondary}, \ref{lemma_U0_VN}, and Theorem \ref{theorem_type1}, respectively. 
\end{proof}

For $t(S) \geq 2$, the inequality $4t(S) +5 \geq 2 t(S) + 9$ holds, so Corollary \ref{corollary_type_main} implies Theorem \ref{theorem_type_main} in this case. Recall that $t(S) = 1$ is equivalent to $S$ being symmetric \cite{Kunz1970}. This allows us to use the following result to resolve the case $t(S) = 1$.

\begin{thm}
    \textup{\cite{Bresinsky1975sym}} If $S$ is symmetric with $e(S)=4$, then $\eta(S) \in \{3,5\}$.
\end{thm}

Since $5 \leq 9 = 4 \cdot 1 + 5$, this completes the proof of Theorem \ref{theorem_type_main}.

\section{Closing remarks}
The bounds of Theorem \ref{theorem_type_main} can most likely be improved. To improve the bound $\eta(S) \geq t(S)-11$, we propose the following conjecture.

\begin{conj}\label{conjecture}
    If $S$ is a numerical semigroup with $e(S) = 4$, then $\eta(S) \geq t(S) +1$.
\end{conj}
 
We expect that $\eta(S) \geq t(S) +1$ is the best possible bound, since no numerical semigroup with $e(S)=4$ and $\eta(S) \leq t(S)$ is currently known. In contrast, the situation where $\eta(S) = t(S) +1$ is quite common among numerical semigroups with embedding dimension four. For example, $S = \langle 77,85,89,131 \rangle$ has $e(S)=4$, $t(S)=3$, and $\eta(S)=4$. Likewise, $S = \langle 76,99,143,223 \rangle$ has $e(S)=4$, $t(S)=6$, $\eta(S)=7$, and $S = \langle 195,200,209,226 \rangle$ has $e(S)=4$, $t(S)=8$, $\eta(S)=9$. The equation $\eta(S) = t(S) +1$ is also satisfied for all numerical semigroups with $e(S)=3$ (Theorem \ref{theorem_embd3}).

\begin{remark}
    In higher embedding dimension, the inequality $\eta(S) \geq t(S) +1$ is false in general, already for embedding dimension five. For example, 
    $S=\langle 86,133,143,155,198 \rangle$ has $e(S)=5$ and $t(S)=\eta(S)=10$. Likewise, $S=\langle 161,177,196,207,225 \rangle$ has $e(S)=5$, $t(S)= 18$, $\eta(S)=16$, and (a more extreme example) $S = \langle 6072, 6307,7965, 8702, 9423 \rangle$ has $e(S)=5$, $t(S)=53$, $\eta(S)=40$.
\end{remark}

\begin{remark}
    The results of Section \ref{section_geometricprocedure}, together with the results of Section \ref{section_minimal_presentations} concerning primary semigroups, appeared earlier in the preprint \cite{Chomicz}.
\end{remark}

\section{Acknowledgments}
At the time of writing this paper, the author was a first-year student at Jagiellonian University in Kraków. The author thanks his tutor,
Tomasz Kowalczyk, for his support.
The author also thanks Alessio Moscariello and Alessio Sammartano for their comments.

\bibliographystyle{plain}
\bibliography{refs.bib}

@book{Rosales2009,
  author    = {J. C. Rosales and P. A. Garc{\'i}a-S{\'a}nchez},
  title     = {Numerical Semigroups},
  series    = {Developments in Mathematics},
  volume    = {20},
  publisher = {Springer},
  address   = {New York},
  year      = {2009}
}

@article{AGUILOGOST2015,
title = {{On the number of L-shapes in embedding dimension four numerical semigroups}},
journal = {Discrete Math.},
volume = {338},
number = {12},
pages = {2168-2178},
year = {2015},
issn = {0012-365X},
doi = {https://doi.org/10.1016/j.disc.2015.05.019},
url = {https://www.sciencedirect.com/science/article/pii/S0012365X15001922},
author = {F. Aguiló-Gost and P. A. García-Sánchez and D. Llena}
}

@article{Rosales1996MinimalRelation,
  author  = {Rosales, J. C.},
  title   = {{An algorithmic method to compute a minimal relation for any numerical semigroup}},
  journal = {Int. J. Algebra Comput.},
  volume  = {6},
  number  = {4},
  year    = {1996},
  pages   = {441--455}
}

@book{AssiDAnnaGarciaSanchez2020,
  author    = {Assi, A. and D'Anna, M. and Garc{\'i}a{-}S{\'a}nchez, P. A.},
  title     = {Numerical Semigroups and Applications},
  series    = {RSME Springer Series},
  volume    = {3},
  edition   = {2},
  year      = {2020},
  publisher = {Springer},
  isbn      = {978-3-030-54943-5},
  doi       = {10.1007/978-3-030-54943-5},
  url       = {https://link.springer.com/book/10.1007/978-3-030-54943-5}
}

@article{Bresinsky1988,
  author    = {Bresinsky, H.},
  title     = {{Binomial generating sets for monomial curves, with applications in $\mathbb{A}^4$}},
  journal   = {Rend. Sem. Mat. Univers. Politecn. Torino},
  volume    = {46},
  pages     = {353-370},
  year      = {1988}
}

@InProceedings{MoscarielloSammartano2025,
  author    = {Moscariello, A. and Sammartano, A.},
  title     = {Open Problems on Relations of Numerical Semigroups},
  booktitle = {Recent Progress in Ring and Factorization Theory},
  pages = {365-380},
  year      = {2025},
  address={Cham},
  publisher = {Springer Nature Switzerland}
}

@article{Chomicz,
    author = {Chomicz, K.} ,
    title = {{On numerical semigroups with embedding dimension four}} ,
    note = {arXiv:2604.25653v3}
}

@article{Kunz1970,
  author  = {Kunz, E.},
  title   = {The value-semigroup of a one-dimensional {G}orenstein ring},
  journal = {Proc. Amer. Math. Soc.},
  volume  = {25},
  year    = {1970},
  pages   = {748-751}
}

@article{Bresinsky1975sym,
  author  = {Bresinsky, H.},
  title   = {Symmetric semigroups of integers generated by 4 elements},
  journal = {Manuscripta Math.},
  volume  = {17},
  year    = {1975},
  pages   = {205-219}
}

@article{Elmacioglu2024,
  author  = {Elmacioglu, C. and Hilmer, K. and O{\textquoteright}Neill, C. and Okandan, M. and Park{-}Kaufmann, H.},
  title   = {On the cardinality of minimal presentations
of numerical semigroups},
  journal = {Algebraic Combinatorics},
  volume  = {7},
  number = {3},
  year    = {2024},
  pages   = {753-771}
}

@article{Froberg1987,
  author  = {Fr{\"o}berg, R. and Gottlieb, C. and H{\"a}ggkvist, R.},
  title   = {On numerical semigroups},
  journal = {Semigroup Forum},
  year    = {1987},
  volume  = {35},
  pages   = {63-83}
}

@article{MoscarielloSammartanominpresen2025,
author = {Moscariello, A. and Sammartano, A.},
title = {On minimal presentations of numerical monoids},
journal = {Bull. Lond. Math. Soc.},
volume = {57},
number = {3},
pages = {878-894},
doi = {https://doi.org/10.1112/blms.70005},
year = {2025}
}

@InProceedings{Stamate2018,
author="Stamate, D. I.",
title="Betti Numbers for Numerical Semigroup Rings",
booktitle="Multigraded Algebra and Applications",
year="2018",
publisher="Springer International Publishing",
address="Cham",
pages="133--157"
}

@article{BarucciDobbsFontana1997,
  author  = {Barucci, V. and Dobbs, D. E. and Fontana, M.},
  title   = {Maximality Properties in Numerical Semigroups and Applications to One-Dimensional Analytically Irreducible Local Domains},
  journal = {Memoirs of the Amer. Math. Soc.},
  volume  = {598},
  year    = {1997}
}

@article{Herzog1970,
  author  = {Herzog, J.},
  title   = {Generators and relations of abelian semigroups and semigroup rings},
  journal = {Manuscripta Math.},
  volume  = {3},
  pages = {175-193},
  year    = {1970}
}

@article{Bresinsky1975onprimeideals,
  author  = {Bresinsky, H.},
  title   = {On prime ideals with generic zero $x_i = t^{n_i}$},
  journal = { Proc. Amer. Math. Soc.},
  volume  = {47},
  pages = {329-332},
  year    = {1975}
}

\vspace{5pt}



\begin{small}

\noindent
Kazimierz Chomicz

\noindent
Institute of Mathematics

\noindent
Faculty of Mathematics and Computer Science

\noindent
Jagiellonian University

\noindent
ul. Łojasiewicza 6, 30-348 Kraków, Poland

\noindent
e-mail: kazimierz.chomicz@student.uj.edu.pl

\end{small}

\end{document}